\documentclass[11pt]{article}
\usepackage{url}
\usepackage{amsmath}
\usepackage{amssymb}
\usepackage{amsthm,enumitem}
\usepackage{empheq}
\usepackage{latexsym}
\usepackage{appendix}
\usepackage{color} 
\usepackage[utf8]{inputenc}
\usepackage[T1]{fontenc}
\DeclareUnicodeCharacter{014D}{\=o}
\usepackage{geometry}
\newtheorem{Theorem}{Theorem}[part]
\newtheorem{Definition}{Definition}[part]

\newtheorem{Assumption}{Assumption}[part]
\newtheorem*{Assumptionn}{Assumption}
\newtheorem{Lemma}{Lemma}[part]
\newtheorem{Corollary}{Corollary}[part]
\newtheorem{Remark}{Remark}[part]

\newtheorem{Ansatz}{Ansatz}
\newtheorem*{sproof}{\textit{Sketch of the proof}}

\makeatletter \@addtoreset{equation}{section}

\@addtoreset{Definition}{section}

\@addtoreset{Theorem}{section}

\@addtoreset{Proposition}{section}

\@addtoreset{Property}{section}

\@addtoreset{Assumption}{section}

\@addtoreset{Corollary}{section}

\@addtoreset{Lemma}{section}

\@addtoreset{Remark}{section}

\@addtoreset{Example}{section}

\@addtoreset{Condition}{section}

\newcommand{\No}[1]{\left\|#1\right\|}     

\def \E{\mathbb{E}}
\def \F{\mathbb{F}}

\def \P{\mathbb{P}}

\def \R{\mathbb{R}}

\def\Cc{{\cal C}}

\def\Mc{{\cal M}}
\def\Nc{{\cal N}}

\def\Pc{{\cal P}}

\def\esup{{\rm ess \, sup}}

\def\eps{\varepsilon}

 \title{A tale of a Principal and many many Agents\footnote{The authors gratefully acknowledge the support of the ANR project Pacman, ANR-16-CE05-0027. This project was started while the three authors were visiting the department of statistics and applied probability at University of California Santa Barbara, whose hospitality is gratefully acknowledged.}}

 \author{Romuald {\sc Elie} \thanks{Universit\'e Paris--Est Marne--la--Vall\'ee, \texttt{romuald.elie@univ-mlv.fr}.} \and Thibaut {\sc Mastrolia} \footnote{CMAP, \'Ecole Polytechnique, Universit\'e Paris Saclay, Palaiseau, France, \texttt{thibaut.mastrolia@polytechnique.edu}. Part of this work was carried out while this author was working at Universit\'e Paris--Dauphine, PSL Research University, CNRS, CEREMADE, 75016 Paris, France, whose support is kindly acknowledged.  The Chair Financial Risks (Risk Foundation, sponsored by Soci\'et\'e G\'en\'erale) is also acknowledged for financial support.}\and Dylan {\sc Possama\"{i}} \footnote{Universit\'e Paris--Dauphine, PSL Research University, CNRS, CEREMADE, 75016 Paris, France, \texttt{possamai@cere-} \texttt{made.dauphine.fr}.}}

\begin{document}

\maketitle

\begin{abstract}
\noindent In this paper, we investigate a moral hazard problem in finite time with lump--sum and continuous payments, involving infinitely many Agents with mean field type interactions, hired by one Principal. By reinterpreting the mean--field game faced by each Agent in terms of a mean field forward backward stochastic differential equation (FBSDE for short), we are able to rewrite the Principal's problem as a control problem of McKean--Vlasov SDEs. We review one general approaches to tackle it, introduced recently in \cite{bayraktar2016randomized,pham2016linear,pham2015bellman,pham2015discrete,pham2016dynamic} using dynamic programming and Hamilton--Jacobi--Bellman (HJB for short) equations, and mention a second one based on the stochastic Pontryagin maximum principle, which follows \cite{carmona2015forward}. We solve completely and explicitly the problem in special cases, going beyond the usual linear--quadratic framework. We finally show in our examples that the optimal contract in the $N-$players' model converges to the mean--field optimal contract when the number of agents goes to $+\infty$, thus illustrating in our specific setting the general results of \cite{cardaliaguet2015master}.

\vspace{5mm}

\vspace{5mm}


\vspace{0.5em}
\end{abstract}

\section{Introduction}
To quote Mary Kom\footnote{Interview in \textit{Sportskeeda}, July 21th, 2012} before the 2012 summer olympic games, "the incentive of a medal at the biggest sporting arena in the world is what drives me. Before I hang my gloves, I want to win the Olympic medal, and my performance at London will decide my future in the sport." When it comes to make an effort in order to achieve a financial, economical, political or personal project, the notion of incentives becomes fundamental and strongly determines the dynamic of the associated project. Many crucial questions in finance and economics are, at their heart, a matter of incentives. How can the government motivate banks and key players in the financial sector to act for the global interest? How can a firm encourage customers to consume (if possible responsibly and economically)? How can an employer motivate her employees to produce more efficiently? 

\vspace{0.5em}
As the above examples make it clear, these questions naturally lead to distinguish between two entities: a \textit{Principal}, who is the one proposing to the second entity, the \textit{Agent}, a contract. The Agent can accept or reject the proposed contract, for instance if he has outside opportunities, but is then committed to it. Of course such a contract will in general involve some kind of effort from the Agent, in exchange of which he receives some sort of compensation, paid by the Principal. The main issue for the Principal is then that, most of the time, he is has only partial information about the actions of the Agent, and does not necessarily observe them. This kind of asymmetry of information leads to a situation which is classically called in the literature "moral hazard". The problem of the Principal is then to design a contract for the Agent, which maximises his own utility, and which of course is accepted by the Agent. This situation is then reduced to finding a Stackelberg equilibrium between the Principal and the Agent. More precisely, one can proceed in two steps to solve these type of problems:
\begin{itemize}[leftmargin=*]
\item given a fixed contract, the Principal computes the best reaction effort of the Agent.
\item Then, the Principal solves his problem by taking into account the best reaction effort of the Agent and computes the associated optimal contract.
\end{itemize} 
This so--called contract theory has known a renewed interest since the seminal paper of Holmstr\"om and Milgrom \cite{holmstrom1987aggregation}, which introduced a convenient method to treat them in a continuous time setting\footnote{There was before that an impressive literature on static or discrete--time models, which we will not detail here. The interested reader can check the references given in the classical books \cite{bolton2005contract,laffont2009theory}}. It was then extended notably by Sch\"attler and Sung \cite{schattler1993first,schattler1997optimal}, Sung \cite{sung1995linearity}, and of course the well--known papers of Sannikov \cite{sannikov2008continuous,sannikov2012contracts}, using an approach based on the dynamic programming and martingale optimality principles, two well--known tools for anyone familiar with stochastic control theory\footnote{There is an alternative approach using the Pontryagin stochastic maximum principle to characterise both the optimal action of the Agent and the optimal contract in terms of fully coupled systems of FBSDEs, initiated by Williams \cite{williams2009dynamic}, see the recent monograph by Cvitani\'c and Zhang \cite{cvitanic2012contract} for more details.}. More recently, this approach has been revisited by Cvitani\'c, Possama\"i and Touzi in \cite{cvitanic2014moral,cvitanic2015dynamic}, where the authors propose a very general approach which not only encompasses, but goes further than the previous literature. They showed that, under an appropriate formulation of the problem,
\begin{itemize}[leftmargin=*]
\item given a fixed contract, the problem of the Agent can always be interpreted in terms of a BSDE when he controls only the drift
, and,

\item the Principal's problem can be rewritten as a standard stochastic control problem with in general two state variables, namely the output controlled by the Agent, and his continuation utility.
\end{itemize}

\noindent One arguably non--realistic aspect of the above literature is that in practice, one would expect that the Principal will have to enter in contracts with several Agents, who may have the possibility of interacting with, or impacting each other. This situation is known in the literature as multi--Agents model. In a one--period framework, this problem was studied by Holmstr\"om \cite{holmstrom1982moral}, Mookherjee \cite{mookherjee1984optimal}, Green and Stokey \cite{green1983comparison}, or Demski and Sappington \cite{demski1984optimal} among others. Extensions to continuous--time models were investigated by Koo, Shim and Sung \cite{keun2008optimal} and very recently by \'Elie and Possama\"i \cite{elie2016contracting}. In a model with competitiveness amongst the Agents, assumed to work for a given firm controlled by the Principal, \cite{elie2016contracting} shows that the general approach initiated by \cite{cvitanic2015dynamic} has a counterpart in the $N$--Agents case as follows 
\begin{itemize}[leftmargin=*]
\item given fixed contracts for all the Agents, finding a Nash equilibria between them can be reduced to finding a solution to a multidimensional BSDE with quadratic growth in general, and,
\item the Principal's problem can once more be rewritten as a standard stochastic control problem with in general $2N$ state variables, namely the outputs controlled by the Agents, and their continuation utilities.
\end{itemize} 
\noindent The main issue here is that unlike with one--dimensional BSDEs, for which there are well established wellposedness theories, multi--dimensional systems are inherently harder. Hence, existence and uniqueness of solutions to systems of quadratic BSDEs have been shown to not always hold by Frei and dos Reis, and and have been subject to in--depth investigations, see for instance Xing and {\v{Z}}itkovi\'c \cite{xing2016class}, or Harter and Richou \cite{harter2016stability}, as well as the references therein. \cite{elie2016contracting} circumvents this problem by imposing wellposedness as a requirement for the admissibility of the contracts proposed by the Principal. Such a restriction would seem, {\it a priori}, to narrow down the scope of the approach, but the authors show that in general situations, the optimal contracts obtained are indeed admissible in the above sense. 

\vspace{0.5em}
\noindent The current paper can be understood as a continuation of \cite{elie2016contracting}, and considers the situation where one lets the number of Agents contracting with the Principal become very large. In terms of applications, one could for instance think about how a firm should provide electricity to a large population, how a government should encourage firms or people to invest in renewable energy by giving fiscal incentives, how city planners should regulate a heavy traffic or a crowd of people... The study of a large number of interacting players is also relevant for the so--called systemic risk theory, which consists in studying financial entities deeply interconnected and strongly subjected to the states of the others, see e.g. Carmona, Fouque and Sun \cite{carmona2013mean2}. All these questions rely on the celebrated \textit{mean field game theory} introduced by Lasry and Lions \cite{lasry2006jeux,lasry2006jeux2,lasry2007mean} and independently by Huang, Caines and Malham\'e \cite{huang2007invariance,huang2006large}. Mean field games theory (MFG for short) consists in the modelling of a large set of identical players who have the same objectives and the same state dynamics. Each of them has very little influence on the overall system and has to take a decision given a mean field term driven by the others. The problem is then to find an equilibrium for the studied system. We refer to the notes of Cardaliaguet \cite{cardaliaguet2010notes} and to the book of Bensoussan, Frehse and Yam \cite{bensoussan2013mean} for nice overviews of this theory. The associated optimal control problems, also coined mean--field type control problems lie at the very heart of our approach and have only been approached recently in the literature, see e.g. Carmona and Delarue \cite{carmona2015forward} or Pham and Wei \cite{pham2015bellman,pham2015discrete,pham2016dynamic}.


\vspace{0.5em}
\noindent Let us now describe in further details our problem. We focus our attention on the situation where the Principal has to hire infinitely many Agents who are supposed to be identical, and who can control the drift of an output process, representing the project that the Agent has to manage on behalf of the Principal. The value of this project is affected through his drift by both its law, and the law of the control, representing both the aggregated impact of the other Agents. We provide a weak formulation of the mean--field version of this model, inspired by \cite{carmona2015probabilistic} and proceed to show that solving the mean field system associated with the Agent's problem is equivalent to solving a kind of mean field BSDE\footnote{Several papers have investigated related but not exactly similar BSDEs, see among others \cite{buckdahn2009mean2,buckdahn2014mean,buckdahn2009mean,carmona2013mean}.}. Even though such objects are actually easier to investigate than the aforementioned systems of quadratic BSDEs, it remains that their wellposedness requires in general very strong assumptions, see \cite{carmona2015probabilistic}. Following in the footsteps of \cite{elie2016contracting}, we therefore embed wellposedness into the definition of an admissible contract. All of this is proved in Section \ref{section:solvingMF}. Once this is done, we are on the right track traced by \cite{cvitanic2015dynamic}, and turn to the Principal's problem in Section \ref{section:mckean}. The latter can then be shown to be equivalent to solving a (difficult) mean--field type control problem with the expected two states variables. We believe this result to be our main technical and practical contribution, since this paper is the first, to the best of our knowledge, in the literature to solve general mean--field Principal--Agent problems\footnote{Let us nonetheless mention the recent contribution of Djehiche and Hegelsson \cite{djehiche2015principal} which considers a related but different problem with only one Agent whose output solves a McKean--Vlasov SDE. Their approach relies on a stochastic maximum principle.}. This opens up the way to considering extremely rich situations of interactions between many Agents, with subtle ripple effects.

\vspace{0.5em}
 \noindent We next apply in Section \ref{section:ex} our general results to specific examples. More precisely, we consider an extension of Holmstr\"om and Milgrom \cite{holmstrom1987aggregation} to the mean--field case, where the drift of the output process is a linear function of the effort of the Agent, the output itself, its mean, its variance and the mean of the effort of the Agent. In particular, we show that the optimal effort is deterministic, by providing a \textit{smooth} solution to the HJB equation associated with the Principal's problem. We then extend this example to the case where the Principal is no longer risk--neutral, and adopts a mean--variance attitude, in the sense that his criterion is impacted by both the variance of the output (he does not want projects with volatile values) and the variance of the salary given to the Agent (he does not want to create discrimination inside the firm). Our final point concerns the rigorous links between the $N-$players' model and the mean--field limit. Proving the convergence in a general setting is an extremely hard problem, see for instance the recent article of Cardaliaguet, Delarue, Lasry and Lions \cite{cardaliaguet2015master}. Therefore, we have concentrated our attention to the above examples, and showed in this context that the optimal contract in the $N-$Agents' model, as well as their optimal actions, indeed converged to the mean--field solution.

\vspace{1em}
\noindent {\bf General notations: } Let $\R$ be the real line, $\mathbb R_+$ the non--negative real line and $\mathbb R_+^\star$ the positive real line. Let $m$ and $n$ be two positive integers. We denote by $\mathcal M_{m,n}(\mathbb R)$ the set of matrices with $m$ rows and $n$ columns, and simplify the notations when $m=n$, by using $\mathcal M_{n}(\mathbb R):=\mathcal M_{n,n}(\mathbb R)$. We denote by ${\rm I}_n\in \mathcal M_n(\R)$ the identity matrix of order $n$. For any $M\in \mathcal M_{m,n}(\mathbb R)$, we define $M^\top \in \mathcal M_{n,m}$ as the usual transposition of the matrix $M$. For any $x\in \mathbb R^n$, we set $\text{diag}(x)\in \mathcal M_{n}(\R)$ such that $(\text{diag}(M))^{i,j}= \mathbf 1_{i=j} x^i$, $1\leq i,j\leq n$. We denote by $\mathbf 1_{n,n}$ the matrix with coefficients $(\mathbf 1_{n,n})^{i,j}=1$ for any $1\leq i,j\leq N$. We will always identify $\mathbb R^n$ with $\mathcal M_{n,1}(\mathbb R)$. Besides, for any $X\in\mathbb R^n$, we denote its coordinates by $X^1,\dots,X^n$. We denote by $\|\cdot\|_n$ the Euclidian norm on $\mathbb R^n$, which we simplify to $|\cdot|$ when $N=1$. The associated inner product between $x\in\mathbb R^n$ and $y\in\mathbb R^n$ is denoted by $x\cdot y$. We also denote by $\mathbf 0_n$ and $\mathbf 1_n$ the $n-$dimensional vector $(0,\dots,0)^\top$ and $(1,\dots,1)^\top$ respectively and $(e_i)_{1\leq i\leq N}$ the canonical basis of $\mathbb R^N$. Similarly, for any $X\in\mathbb R^n$, we define for any $i=1,\dots,n,$ $X^{-i}\in\mathbb R^{n-1}$ as the vector $X$ without its $i$th component, that is to say $X^{-i}:=(X^1,\dots,X^{i-1},X^{i+1},\dots,X^n)^\top$. For any $X,Y$ in $\R^N$ we will write $X\leq Y$ for the classical lexicographic order \textit{i.e.} $X\leq Y$ if $X^i\leq Y^i$, $\forall 1\leq i\leq N$. Finally, for any $(a,\tilde a)\in \mathbb R\times\mathbb R^{n-1}$, and any $i=1,\dots,n$, we define the following $n-$dimensional vector
$$a\otimes_i \tilde a:= (\tilde a^1,\dots, \tilde a^{i-1}, a, \tilde a^i ,\dots, \tilde a^{n-1}).$$

For any Banach space $(E,\|\cdot \|_E)$, let $f$ be a map from $E\times \mathbb R^n$ into $\mathbb R$. For any $x\in E$, we denote by $\nabla_a f(x,a)$ the gradient of $a\longmapsto f(x,a),$ and we denote by $\partial_{a a} f(x,a)$ the Hessian matrix of $a\longmapsto f(x,a)$. When $n=1$, we write $f_a(x,a)$ for the derivative of $f$ with respect to the variable $a$.

 Finally, we denote by $\mathcal D$ the set of deterministic function from $[0,T]$ into $\R$.

\section{The mean field problem}\label{intro:MFM}
This section is dedicated to the description of the mean--field problem studied in this investigation. The paradigme considered is the following, we consider an entire crowd of Agents hired by one Principal and we focus on a representative one interacting with the theoretical distribution of the infinite number of other players. 

\subsection{Stochastic basis and spaces} 
In this section we recall some notations used in \cite{carmona2015probabilistic}. Fix a positive integer $N$ and a positive real number $T$.
 For any measurable space $(\mathcal S, \mathcal{F}_{\mathcal S})$, we will denote by $\mathcal P (\mathcal \mathcal S)$ the set of probability measures on $\mathcal S$. Let $(E,\| \cdot\|_E)$ be a Banach space. We will always refer to the Borel sets of $E$ (associated to the topology induced by the norm $\| \cdot\|_E$) by $\mathcal B(E)$. We will also endow this set with the topology induced by the weak convergence of probability measures, that is to say that a sequence $(m_n)_n$ in $\mathcal P(E)$ converges weakly to $m\in \mathcal P(E)$ if for any bounded continuous map $\varphi: E\longrightarrow E$, we have
$$\lim_{n\to +\infty} \int_E \varphi(x) dm_n(x)=\int_E \varphi(x) dm(x). $$ 
This convergence is associated to the classical Wasserstein distance of order $p\geq 1$, defined for any $\mu$ and $\nu$ in $\mathcal P(E)$ by
$$\mathcal W_{E,p}(\mu,\nu)=\left(\inf_{\pi \in \Gamma(\mu,\nu)}\int_E \|x-y\|^p_E\pi(dx,dy)\right)^\frac1p, $$
where $\Gamma(\mu,\nu)$ denotes the space of all joint distributions with marginal laws $\mu$ and $\nu$. More precisely, convergence in the Wasserstein distance of order $p$ is equivalent to weak convergence plus convergence of the first $p$th moments.

\vspace{0.5em}
Let $\mathcal C:=\mathcal C([0,T];\mathbb R)$ be the space of continuous maps from $[0,T]$ into $\mathbb R^N$, endowed with the norm $\| \omega\|_{T,\infty}$, where for any $t\in [0,T]$, we have defined $\| \omega\|_{t,\infty}=\sup_{s\in [0,t]} \|\omega_s\|.$ We denote $\Omega:= \mathbb R^N\times \mathcal C$. 
and define the coordinate processes in the space $\Omega$ by
$$\psi(x,\omega):= x,\; W(x,\omega):= \omega,\; \forall (x,\omega)\in \Omega. $$
 
We fix a probability measure $\lambda _0$ in $\mathcal P(\mathbb R)$, which will serve as our initial distribution for the state of the different Agents of the model. We will always assume that $\lambda_0$ has exponential moments of any order, that is to say
\begin{equation}\label{eq:lambda0}
\int_{\mathbb R^N}\exp(px)\lambda_0(dx)<\infty,\ \text{for any $p\geq 0$}.
\end{equation}
We denote by $\mathbb P$ the product of $\lambda_0$ with the Wiener measure defined on $\mathcal B(\mathcal C)$. For any $t\in [0,T]$, we define $\mathcal F_t$ as the $\mathbb P-$augmentation of the $\sigma-$field $\sigma((\psi,W_s)_{s\in [0,t]})$, as well as the filtration $\mathbb F:= (\mathcal F_t)_{t\in [0,T]}$. The filtered probability space we will be interested in is $(\Omega, \mathcal B(\Omega),\mathbb F,\mathbb P)$. Expectations or conditional expectations under $\P$ will be denoted by $\E[\cdot]$ and $\E[\cdot|\cdot]$.We also denote by $\mathcal T_{[0,T]}$ the set of $\mathbb F-$stopping times which take value in $[0,T]$.
 
 \vspace{0.5em}
For any finite dimensional normed space $(E,\No{\cdot}_E)$, $ \mathcal P_{roc}(E)$ (resp. $\mathcal P_{rev}(E)$) will denote the set of $E-$valued, $\mathbb F-$adapted processes (resp. $\mathbb F-$predictable processes) and for any $p\geq 1$ and $\ell>0$
 \begin{align*}
 \mathbb S^p(E)&:= \left\{Y\in \mathcal P_{roc}(E),\ \text{c\`adl\`ag, such that } \| Y\|_{\mathbb S^p(E)}^p:=  \mathbb E\bigg [\sup_{t\in [0,T]}\No{Y_t}_E^p\bigg ]<+\infty\right\},\\
 \mathbb S_{\exp}(E)&:= \left\{Y\in  \mathcal P_{roc}(E),\; \text{c\`adl\`ag and such that }  \mathbb E\left[\exp\left(p\underset{0\leq t\leq T}{\sup}\|Y_t\|_E \right) \right]<+\infty,\; \forall p\geq 1 \right\},\\
  \mathbb H^p(E)&:= \left\{Z\in \mathcal P_{rev}(E),\; \| Z\|_{\mathbb H^p}^p:= \mathbb E\left[\left(\int_0^T \|Z_t\|_E^2dt\right)^{p/2}\right]<+\infty\right\},\\
  \mathbb H_{\rm exp}^{\ell}(E)&:= \left\{Z\in \mathcal P_{rev}(E),\; \mathbb E\left[\exp\left(m\int_0^T \|Z_t\|_E^\ell dt \right) \right]<+\infty,\; \forall m\geq 1 \right\}.
 \end{align*}
 
 \vspace{0.5em}
In this paper we will almost always consider $E-$valued processes $X:[0,T]\times\mathcal C\longrightarrow E$ which will be $\F-$optional. Notice that such processes are automatically non--anticipative in the sense that for any $(x,x')\in\mathcal C\times\mathcal C$ such that for some $t\in[0,T]$, $x_{\cdot\wedge t}=x'_{\cdot\wedge t}$, we have
$$X(s,x)=X(s,x'),\ \text{for any $s\leq t$}.$$
Fix some $\sigma-$algebra $\mathcal G$ on $(\Omega,\mathcal B(\Omega)$. For any $E-$valued and $\mathcal G-$measurable random variable $F$ on $\Omega$, we denote by $\mathcal L(F):= \mathbb P\circ F^{-1} \in\mathcal P(\Omega)$ the law of $F$ under $\P$ and for any $p>0$, we set
$$L^p(\Omega,\mathcal G,E):= \left\{F:\Omega\longrightarrow E, \ \text{$\mathcal G-$measurable, s.t. }\mathbb E[\|F\|_E^p]<+\infty \right\}. $$

\subsection{Output, controls, drift, volatility, cost and discount factor}
\noindent The problems that we will consider necessitate to introduce a certain number of maps.
We first introduce a volatility process $\sigma:[0,T]\times\mathcal C\longrightarrow \R\backslash\{0\}$. It will be assumed to satisfy the following.

\begin{Assumptionn}[$\boldsymbol{\sigma}$]\label{assum:ssigma}
The map $\sigma$ is bounded by some positive constant $M$, $\mathbb F-$optional, and for every $(t,x)\in[0,T]\times\mathcal C$, $\sigma(t,x)$ is invertible with inverse bounded by some positive constant $M$. Moreover, $\sigma$ is such that the following stochastic differential equation admits a unique strong solution
$$X_t=\psi+\int_0^t\sigma_s(X)dW_s,\ t\in[0,T],\ \P-a.s.$$
 \end{Assumptionn}
Let us define the following set
$$\mathfrak P(\mathbb R):=\left\{q:[0,T]\longrightarrow\mathcal P(\R),\ \text{measurable}\right\}.$$

\noindent Our second object will be the so--called drift function, which drives the value of the output. It will be a map $b$ from $[0,T] \times \mathcal C\times \mathcal P(\mathcal C)\times \mathcal P(\mathbb R) \times  \mathbb R$ into $\mathbb R$. For any $(p,\ell,\eta)\in [1,+\infty)\times (0,+\infty)\times (1,+\infty)$, we will consider the following assumption on $b$
\begin{Assumptionn}[$\mathbf{B}^{p,\ell,\eta}$]\label{hyp:bN}

\vspace{0.5em}
$(i)$ For any $(\mu,q,a)\in \mathcal P(\mathcal C)\times \mathcal P(\mathbb R)\times \R $,
the map $(t,x)\longmapsto b(t,x,\mu,q,a)$ is $\F-$optional.

\vspace{0.5em}
$(ii)$ For any $(t,x,\mu,q)\in [0,T]\times \mathcal C \times  \mathcal P(\mathcal C)\times \mathcal P(\mathbb R)$, the map $a\longmapsto b(t,x,\mu,q,a)$ is continuously differentiable on $\mathbb R$. 

\vspace{0.5em}
$(iii)$ There is a positive constant $C$ such that for any $(t,x,\mu,q,a) \in [0,T]\times  \mathcal C \times  \mathcal P(\mathcal C)\times \mathcal P(\mathbb R) \times \R$, we have
$$ |b(t,x,\mu,q,a)|\leq b^0(\|x\|_{t,\infty})+ C\bigg( 1 +\left(\int_{\mathcal C} \|z\|_{t,\infty}^p \mu(dz)\right)^{\frac1p}+\left(\int_{\mathbb R} |z|^p q(dz)\right)^{\frac1p}+|a|^\ell\bigg),$$
$$|\partial_ab(t,x,\mu,q,a)|\leq C\bigg( 1+b^1(\|x\|_{t,\infty})  +\left(\int_{\mathcal C} \|z\|_{t,\infty}^p \mu(dz)\right)^{\frac1p}+\left(\int_{\mathbb R} |z|^p q(dz)\right)^{\frac1p} +|a|^{\ell-1}\bigg),$$
where $b^0,b^1: \mathbb R_+  \longrightarrow \mathbb R_+$ are such that 
$$\mathbb E\left[\exp\left(\frac{3\eta}{2} M^2\int_0^T|b^0(\| X\|_{t,\infty})|^2 dt\right)+\exp\left(h \int_0^T|b^1(\| X\|_{t,\infty})|^{h'} dt\right)\right]<+\infty, \, \forall h,h'>1, $$
with $M$ denoting the constant defined in Assumption $(\boldsymbol{\sigma})$. \end{Assumptionn}

 \begin{Remark}
Letting the drift $b$ depend on the laws of the output and the control allows for instance to incorporate the following effects.
\begin{itemize}[leftmargin=*]
\item[$(i)$] The law value of the firm can have a positive or negative impact on its future evolution: when things are going well, they have a tendency to keep doing so, and conversely. 
\item[$(ii)$] Similarly, if all the crowd of Agents in the firm are working on average very hard, this could have a ripple effect on the whole firm. Hence a dependence on the law of the control itself.
\end{itemize} 
Furthermore, the growth assumptions made here are for tractability, and are mainly due to the fact that we are going at some point to work with quadratic $($of possibly mean field type$)$ BSDEs. They allow for instance for drifts which are linear in $x$, with a constant small enough. Indeed, in this case $X$ is basically like a Brownian motion $($recall that $\sigma$ is bounded$)$ so that its square will have exponential moments, provided that they are of a small order.
\end{Remark}
 The third object will be a discount factor $k$, that is a map from $[0,T]\times \mathcal C\times \mathcal P(\mathcal C)\times  \mathcal P(\mathbb R)$ to $\mathbb R$, satisfying the following standing assumption.
\begin{Assumptionn}[$\mathbf{K}$]\label{hyp:k} The map
$k$ is bounded, and $\F-$optional for any $(\mu,q)\in \mathcal P(\mathcal C)\times\mathcal P(\mathbb R)$.
\end{Assumptionn}
\begin{Remark}
The discount factor is here to model a possible impatience of the Agents, who would value more having utility now than later. Letting it depend on the law of $X$ and the law of the control played by the Agents is again for generality and possible ripple effects. Agents may for instance become more impatient if the firm is currently doing extremely well. We could have let $k$ also depend on the control itself, but it would have further complicated our definition of admissible controls. We therefore decided to refrain from it, but it is not a limit of the theory, {\it per se}.
\end{Remark}

Finally, we will need to consider a cost function $c:[0,T]\times \mathcal C\times\mathcal P(\mathcal C)\times\mathcal P(\mathbb R)\times \mathbb R \longrightarrow \mathbb R^+$ satisfying, for some $(p,\ell,m,\underline m)\in[1,+\infty)\times(0,+\infty)\times[\ell,+\infty) \times(\ell-1,+\infty)$
\begin{Assumptionn}[$\mathbf{C}^{p,\ell,m,\underline m}$]\label{hyp:cN}
For any $(\mu,q,a)\in \mathcal P(\mathcal C)\times \mathcal P(\mathbb R)\times \R$, the map $(t,x)\longmapsto c(t,x,\mu,q,a)$ is $\F-$optional. Moreover, the map $a\longmapsto c(t,x,\mu, q,a)$ is increasing, strictly convex and continuously differentiable for any $(t,x,\mu,q)\in [0,T]\times \mathcal C\times \mathcal P(\mathcal C)\times\mathcal P(\mathbb R)$.
Finally, there exists $C>0$ such that for any $(t,x,\mu,q,a)\in [0,T]\times \mathcal C\times\mathcal P(\mathcal C)\times\mathcal P(\mathbb R)\times \R$
$$0\leq c(s,x,\mu,q,a)\leq C \bigg(1+\|x\|_{s,\infty}+\left(\int_{\mathcal C} \|z\|_{s,\infty}^p \mu(dz)\right)^{\frac1p}+\left(\int_{\mathbb R^N} |z|^p q(dz)\right)^{\frac1p}+|a|^{\ell+m}\bigg), $$
$$ |\partial_ac(s,x,\mu,q,a)|\geq C  |a|^{\underline m}, \text{ and } \overline{\lim}_{|a|\to \infty} \frac{c(s,x,\mu,q,a)}{|a|^\ell}=+\infty.$$

\end{Assumptionn}
\begin{Remark}
Once again, the cost faced by the Agents can be influenced by the past states of the firm and its law, as well as the law of the control played by the other Agents. It helps to model the fact that Agents may find it hard to work when everyone else is working $($the classical free rider problem$)$, or if the situation of the firm is on average extremely good. As for the growth conditions assumed, they are basically here to ensure that the Hamiltonian of the Agent, which will involve both $b$ and $c$, has at least one maximiser in $a$, thanks to nice coercivity properties.
\end{Remark}

We now turn to the definition of the notion of admissible control. We denote by $\mathcal A$ the set of $\R-$valued and $\F-$adapted processes $\alpha$ such that for some $\varepsilon>0$, for every $(h,\mu,q)\in\R_+\times\Pc(\Cc)\times\mathfrak P(\R)$ and for the same $m$ and $\ell$ as the ones appearing in Assumption $(\mathbf{C}^{p,\ell,m,\underline m})$
\begin{equation}\label{eq:admm}
\mathbb E\left[\left(\mathcal E\left( \int_0^T   \sigma_t^{-1}(X) b(t, X, \mu, q_t,\alpha_t)dW_t \right)\right)^{1+\eps}\right]+\mathbb E\left[\exp\left(h\int_0^T|\alpha_t|^{\ell+m}dt\right)\right]<+\infty.
\end{equation}
Then, for any $(\mu,q,\alpha)\in\mathcal P(\Cc)\times\mathfrak P(\R)\times \mathcal A$, we can define, thanks to Assumption \ref{hyp:bN}, a probability $\mathbb P^{\mu,q,\alpha}$ such that
\begin{equation}\label{def:probastar}
\frac{d\mathbb P^{\mu,q,\alpha}}{d\mathbb P}= \mathcal E\left( \int_0^T \sigma_t^{-1}(X)b(t,X, \mu, q_t, \alpha_t)dt\right).
\end{equation}
Hence, we can define a Brownian motion under $\mathbb P^{\mu,q,\alpha}$ by
$$W^{\mu,q,\alpha}_t:= W_t-\int_0^t  \sigma_s^{-1}(X)b(s,X,\mu, q_s, \alpha_s)ds,\ t\in[0,T],\ \P-a.s.,$$
so that we can rewrite 
 $$X_t=\psi+\int_0^t b(s,X,\mu,q_s,\alpha_s) ds+ \int_0^t\sigma_s(X) dW^{\mu,q,\alpha}_s,\ t\in[0,T],\ \P-a.s. $$

\subsection{The Agent's problem: a classical mean--field game}  

We consider that a representative Agent has utility functions $U_A:\mathbb R\longrightarrow \mathbb R$ and $u_A:[0,T]\times\mathcal C\times\mathcal P(\mathcal C)\times\mathcal P(\mathbb R)\times\mathbb R_+\longrightarrow\mathbb R$. We will assume

\begin{Assumptionn}[$\mathbf U$]\label{assump:utility}
The map $U_A$ is non--decreasing and concave, and for any $(t,x,\mu,q)\in[0,T]\times\mathcal C\times\mathcal P(\mathcal C)\times\mathcal P(\mathbb R)$, the map $\chi\longmapsto u_A(t,x,\mu,q,\chi)$ is non--decreasing and concave. Furthermore, for any $(\mu,q,\chi)\in\mathcal P(\mathcal C)\times\mathcal P(\mathbb R^N)\times\mathbb R_+$, the map $(t,x)\longmapsto u_A(t,x,\mu,q,\chi)$ is $\F-$optional, and there is a positive constant $C$ and a concave map $\tilde u_A:\mathbb R_+\longrightarrow \mathbb R_+$ such that for any $(t,x,\mu,q,\chi) \in [0,T]\times  \mathcal C \times  \mathcal P(\mathcal C)\times \mathcal P(\mathbb R^N) \times \R_+$, we have
$$ |u_A(t,x,\mu,q,\chi)|\leq C\bigg( 1+\| x\|_{t,\infty} +\bigg(\int_{\mathcal C} \|z\|_{t,\infty}^p \mu(dz)\big)^{\frac1p}+\bigg(\int_{\mathbb R} |z|^p q(dz)\bigg)^{\frac1p}+\tilde u_A(\chi)\bigg).$$

\end{Assumptionn}
The representative Agent is hired at time $0$ by the Principal on a "take it or leave it" basis. The Principal proposes a contract to the Agent, which consists in two objects.
\begin{itemize}[leftmargin=*]
\item[$(i)$] A stream of payment $\chi$, which is a $(\R_+^\star)-$valued and $\mathbb F-$adapted process, such that $\chi_t$ represents the instantaneous payments made to the Agent at time $t$.
\item[$(ii)$] A final payment $\xi$, which is a $\mathbb R-$valued and $\mathcal F_T-$measurable random variable which represents the amount of money received by the Agent at time $T$.
\end{itemize}  
A contract will always refer to the pair $(\chi,\xi)$, and the set of contracts will be denoted by $\mathfrak C$. For given $(\mu,q,\alpha)\in\Pc(\Cc)\times\mathfrak P(\R)\times \mathcal A$, representing respectively an arbitrary distribution of the output managed by the infinitely many other Agents, an arbitrary distribution of the actions chosen by these infinitely many Agents, and an action chosen by the representative Agent, his associated utility is given by
\begin{align*}
v_0^{A}(\chi,\xi,\mu,q,\alpha):=&\ \mathbb E^{\mathbb P^{\mu,q,\alpha}}\left[K_{0,T}^{X,\mu,q}U_{A}(\xi)+ \int_0^T K_{0,s}^{X,\mu,q,\alpha}\big(u_A(s,X, \mu,q_s,\chi_s)-c(s,X, \mu,q_s,\alpha_s)\big)ds \right],
\end{align*}
where for any $(x,\mu,q,t,s)\in\Cc\times\Pc(\Cc)\times\mathfrak P(\R)\times[0,T]\times[t,T]$
$$K_{t,s}^{x,\mu,q}:= \exp\left(-\int_t^s  k(u,x,\mu,q_s)du\right).$$
In another words, the Agent profits from the (discounted) utilities of his terminal payment, and his inter--temporal payments, net of his instantaneous cost of working. Overall, the problem of the Agent corresponds to the following maximisation
\begin{align}
\label{eq:valueagent}V_0^A(\chi,\xi,\mu,q):=\sup_{\alpha\in \mathcal A} v_0^{A}(\chi,\xi,\mu, q,\alpha).
\end{align}

As usual in mean--field games, it is immediate that the best reaction function of the Agent is a standard (albeit non--Markovian) stochastic control problem, where $\mu$ and $q$ only play the role of parameters. As such, it is a well--known result that, since the Agents can only impact the drift of the output process, the study of the dynamic version of the Agent's value function requires to introduce first the following family of BSDEs, indexed by $(\mu,q,\alpha)\in\mathcal P(\mathcal C)\times \mathfrak P(\mathbb R) \times\mathcal A$
\begin{align}\label{edsr:agent:loisfixes}
Y_t^{ \mu,q, \alpha}(\chi,\xi)=&\ U_A(\xi)+\int_t^T g\big(s,X, Y_s^{\mu,q, \alpha}(\chi,\xi), Z_s^{\mu,q,\alpha}(\chi,\xi), \mu,q_s,\alpha_s,\chi_s\big)  ds-\int_t^T Z_s^{ \mu,q, \alpha}(\chi,\xi) \sigma_s(X) dW_s,
\end{align}
where for any $(s,x,y,z,\mu,q,a,\chi)\in [0,T]\times \mathcal C\times \mathbb R\times \mathbb R\times \mathcal P(\mathcal C)\times \mathfrak P(\mathbb R)\times \mathbb R\times\mathbb R_+$ we defined
\begin{equation}\label{def:g}g(s,x,y,z,\mu,q_s,a,\chi):=zb(s,x,\mu,q_s,a)+u_A(s,x,\mu,q_s, \chi)-k(s,x,\mu,q_s)y-c(s,x, \mu,q_s,a) .
\end{equation}
We begin by defining a solution to BSDE \eqref{edsr:agent:loisfixes}.
\begin{Definition}\label{def:sol:edsr:agent:loisfixes}
We say that a pair of processes $(Y^{\mu,q,\alpha}(\chi,\xi),Z^{\mu,q,\alpha}(\chi,\xi))$ solves the BSDE \eqref{edsr:agent:loisfixes} if $Y^{\mu,q,\alpha}(\chi,\xi)\in\mathbb S_{\exp}(\mathbb R)$, $Z^{\mu,q,\alpha}(\chi,\xi)\in\mathbb H^p(\mathbb R)$ for any $p\geq 0$ and \eqref{edsr:agent:loisfixes} holds for any $t\in[0,T]$, $\P-a.s.$
\end{Definition}
We now make the aforementioned link between $v_0^A(\chi,\xi,\mu, q, \alpha)$ and BSDE \eqref{edsr:agent:loisfixes} clear with the following Lemma, \textcolor{black}{whose proof is classical, but which we recall in the Appendix for comprehensiveness.}
\begin{Lemma}\label{lemma:soledsrloisfixes}
Let Assumptions $(\mathbf{B}^{p,\ell,\eta})$, $(\boldsymbol{\sigma})$, $\mathbf{(U)}$, $\mathbf{(K)}$ and $(\mathbf{C}^{p,\ell, m,\underline m})$ be true for some $(p,\ell,m,\underline m,\eta)\in[1,+\infty)\times(1,+\infty)\times[\ell,+\infty)\times (\ell-1,+\infty)\times (1,+\infty)$. For any $(\mu,q,\alpha)\in\mathcal P(\mathcal C)\times \mathfrak P(\mathbb R) \times\mathcal A$, there exists a unique solution $(Y^{\mu,q,\alpha}(\chi,\xi),Z^{\mu,q,\alpha}(\chi,\xi))$ to BSDE \eqref{edsr:agent:loisfixes}. Moreover, it satisfies 
$$\mathbb E\left[Y_0^{\mu,q,\alpha}(\chi,\xi)\right]=v_0^{A}(\chi,\xi,\mu, q, \alpha).$$
\end{Lemma}

Now, we turn our attention to the equilibrium between the Agents, which consists in solving the problem \eqref{eq:valueagent} and finding an associated fixed point. For given $(\chi,\xi)\in \mathfrak C$, we call this problem \textbf{(MFG)}$(\chi,\xi)$ and recall our readers the rigorous definition of a solution, taken from \cite{carmona2015probabilistic}.
\begin{Definition}[Solution of \textbf{$($MFG$)(\chi,\xi)$}]\label{def:solMFG} A triplet $(\mu,q,\alpha)\in \mathcal P(\mathcal C)\times \mathfrak P(\R)\times \mathcal A$ is a solution to the system $({{\rm \bf MFG}})(\chi,\xi)$ if $V_0^{A}(\chi,\xi,\mu,q)=v_0^A(\chi,\xi,\mu,q,\alpha)$, $\mathbb P^{\mu,q, \alpha}\circ (X)^{-1}=\mu $ and $\mathbb P^{\mu,q,\alpha}\circ (\alpha_t)^{-1}=q_t $ for Lebesgue almost every $t\in[0,T]$.
\end{Definition}
Recall from Lemma \ref{lemma:soledsrloisfixes} that for any triplet $(\mu,q,\alpha)\in \mathcal P(\mathcal C)\times \mathfrak P(\mathbb R)\times \mathcal A$, there exists a unique solution to BSDE \eqref{edsr:agent:loisfixes}. The notion of admissibility for an effort given by \eqref{eq:admm} implies implicitly that we have to restrict a bit more the notion of a solution to $\mathbf{(MFG)}(\chi,\xi)$, mainly in terms of the required integrability. Let $r>1$ be fixed throughout the rest of the paper. We denote by ${\rm MF}^r(\chi, \xi)$ the set of solutions $(\mu,q,\alpha)$ to $({{\rm \bf MFG}})(\chi,\xi)$ such that the second component $Z^{\mu,q,\alpha}$ of the solution to BSDE \eqref{edsr:agent:loisfixes} (associated to the solution of \textbf{$($MFG$)(\chi,\xi)$} by Lemma \ref{lemma:soledsrloisfixes}) is in the space $\mathbb H^{\lambda r}_{\rm exp}(\mathbb R),$ with $\lambda:= (\ell+m)/(\underline m+1-\ell)$ (recall that by definition, $\underline m+1> \ell$).

\vspace{0.3em}
\noindent In general, the system \textbf{(MFG)}$(\chi,\xi)$ can admit several solutions, \textit{i.e.} the set ${\rm MF}^r(\chi, \xi)$ is not necessarily reduced to a unique triplet $(\mu,q,\alpha)$. \textcolor{black}{To simplify our study, we will assume that the choice of the equilibrium played by the Agents is handed over to the Principal, so that she chooses both contracts and mean--field equilibria. One interpretation is that the Agents are not sophisticated enough to select a mean--field equilibrium, but are however able to check if the proposed effort is indeed a component of a mean--field equilibrium. }
\begin{Remark}
We would also like to emphasise that the Agents could also be allowed to select one equilibrium in {\rm$\textbf{(MFG)}(\chi,\xi)$} themselves. This framework is also covered by our study, provided that one fixes  that equilibrium in the rest of the paper. Nevertheless, in view of the literature on Principal--Agent problems, it is more relevant to assume that the choice is delegated to the Principal, since classically the Principal offers to the Agent both a contract and a recommended level of effort $($see {\rm \cite{laffont2009theory}} for instance$)$. 
\end{Remark}

\subsection{The Principal problem}\label{section:PP:chaos}

Before defining the problem of the Principal, we need to define the set of admissible contracts. The idea is to consider only contracts such that the Principal is able to compute the reaction of the Agents, that is to say the ones for which there is at least one mean field equilibrium. Although there could arguably be a discussion on the question of whether considering that the Agents are looking for a mean field equilibrium or not is the most pertinent one, we believe that once this choice has been made, our assumption makes sense from the practical point of view. Indeed, the Principal needs to be able to anticipate, one way or another, how the Agents are going to react to the contract that he may offer, and will therefore not offer contract for which Agents cannot agree on an equilibrium. Granted, one could also resort to approximate equilibria or other related notions, but this paper being the first one in the literature treating moral hazard problems with mean--field interactions, we have chosen to work in a setting which remains reasonable and tractable at the same time. This being said, the contract also has to take into account the fact that a representative Agent has a reservation utility $R_0$ and will never accept a contract which does not guarantee them at least that amount of utility. 
\vspace{0.5em}

Finally, we need to add some integrability assumptions, which finally leads us to the set of admissible contracts $\Xi$ defined by the following.

\begin{Definition} The set $\Xi$ is composed by pairs $(\chi,\xi)\in\mathfrak C$, such that 
\begin{itemize}[leftmargin=*]
 \item[$(i)$] ${\rm MF}^r(\chi,\xi)\neq \emptyset$.
 \item[$(ii)$] For any $(\mu^{\star},q^{\star},\alpha^{\star})\in {\rm MF}^r(\chi,\xi) $ we have $V^{A}_0(\chi,\xi, \mu^{\star},q^{\star})\geq R_0$.
 \item[$(iii)$] For any $p\geq 0$
\begin{equation}\label{inegalite:chixi} \mathbb E\left[\exp\left(p\left(|U_A(\xi)|+\int_0^T\tilde u_A(\chi_s)ds\right)\right)\right]<\infty.\end{equation}
\end{itemize}
\end{Definition}

The Principal's problem is then to solve the following optimisation
\begin{equation}\label{PrincipalProblem2}
U_0^P:= \sup_{(\chi,\xi)\in \Xi}\, \sup_{(\mu,q,\alpha)\in {\rm MF}^r(\chi,\xi)}\, \mathbb E^{\P^{\mu,q,\alpha}}\left[U_P\bigg(X_T-\xi-\int_0^T \chi_s ds \bigg)\right],
\end{equation}
where $U_P$ is an increasing map from $\mathbb R$ into $\mathbb R$. As far as we know, such a problem has never been considered so far in the literature. It basically boils down to finding a Stackelberg equilibrium between the Principal and infinitely many Agents in mean--field equilibrium. The main contribution of this paper is to show that we can actually reduce it to the study of a system of controlled McKean--Vlasov SDEs. We will describe in Section \ref{section:mckean} some approaches to try and tackle the latter problem.
\section{Solving the mean--field game: yet another BSDE story}\label{section:solvingMF}
\subsection{Optimal effort}

We now present a result ensuring the existence of a maximiser for $g$ with respect to the effort of the Agent, as well as associated growth estimates.


\begin{Lemma}\label{lemma:maxg}
Let Assumptions $(\mathbf{B}^{p,\ell,\eta})$, $(\boldsymbol{\sigma})$, $\mathbf{(K)}$ and $(\mathbf{C}^{p,\ell,m,\underline m})$ hold true for some $(p,\ell,m,\underline m,\eta)\in[1,+\infty)\times(0,+\infty)\times[\ell,+\infty)\times(\ell-1,+\infty)\times (1,+\infty)$. Then, for any $(s,x,\mu,q,y,z,\chi)\in [0,T]\times \mathcal C\times \mathcal P(\mathcal C)\times \mathfrak P(\mathbb R)\times \mathbb R\times\mathbb R\times\mathbb R_+$ there exists $a^\star(s,x,z,\mu,q)\in \mathbb R$ such that
\begin{equation}\label{eq:g}
a^\star(s,x,z,\mu,q)\in\underset{a\in A}{\rm{argmax }}\,g(s,x, y, z, \mu,q_s,a,\chi),\end{equation}
satisfying for some positive constant $C$
\begin{align*}
&|a^\star(t,x,z,\mu,q)|\\
&\leq C\bigg( 1+ |z|^{\frac{1}{\underline m+1-\ell}}\bigg( 1+|b^1(\No{x}_{t,\infty})|^{\frac{1}{\underline m+1-\ell}} +\bigg(\int_{\mathcal C} \|w\|_{t,\infty}^p \mu(dw)\bigg)^{\frac{1}{p(\underline m+1-\ell)}}+\bigg(\int_{\mathbb R} |w|^p q(dw)\bigg)^{\frac{1}{p(\underline m+1-\ell)}}\bigg) \bigg). 
\end{align*}
Furthermore, we have
\begin{align*} 
&|g(s,x,y,z,\mu,q_s,a^\star(s,x,z,\mu,q),\chi)|\\
&\leq C\bigg(1+\|x\|_{s,\infty}+|y|+|z|^{\frac{(\ell+m)\vee(\underline m+1)}{\underline m+1-\ell}}+\bigg(\int_{\mathcal C} \|w\|_{t,\infty}^p \mu(dw)\bigg)^{\frac1p}+\bigg(\int_{\mathbb R} |w|^p q(dw)\bigg)^{\frac1p}+\tilde u_A(\chi) \bigg)\\
&\hspace{0.9em}+C|z|\bigg(b^0(\No{x}_{t,\infty})+ \bigg(\int_{\mathcal C} \|w\|_{t,\infty}^p \mu(dw)\bigg)^{\frac1p}+\bigg(\int_{\mathbb R} |w|^p q(dw)\bigg)^{\frac1p}\bigg)\\
&\hspace{0.9em}+C|z|^{\frac{\underline m+1}{\underline m+1-\ell}}\bigg(|b^1(\No{x}_{t,\infty})|^{\frac{1}{\underline m+1-\ell}} +\bigg(\int_{\mathcal C} \|w\|_{t,\infty}^p \mu(dw)\bigg)^{\frac{1}{p(\underline m+1-\ell)}}+\bigg(\int_{\mathbb R} |w|^p q(dz)\bigg)^{\frac{1}{p(\underline m+1-\ell)}} \bigg).
\end{align*}
\end{Lemma}
In the following, we set 
$$ g^\star(s,x,y,z,\mu,q_s,\chi):= \sup_{a\in A} \, g(s,x, y, z, \mu,q_s,a\chi). $$

\textcolor{black}{For the sake of simplicity, we will from now on reduce our study to the case where $g$ only admits one maximiser, which actually holds in all our examples studied in Section \ref{section:ex}, an which basically corresponds to assuming sufficient coercivity for the Hamiltonian. Our approach extends to the a more general setting where the Principal can choose among these optimal efforts, albeit with more complicated notations.
\begin{Assumption}\label{assumption:unique} For any $(t,x,y,z,\mu,q,\chi)\in [0,T]\times\mathcal C\times\mathbb R\times\mathbb R\times\mathcal P(\mathcal C)\times\mathfrak P(\mathbb R)\times\mathbb R_+$, the map 
$$a\in A\longmapsto g(t,x, y, z, \mu,q_s,a,\chi),$$ admits a unique measurable maximiser $a^\star(t,x,y,\mu,q,\chi)$. \end{Assumption} }

%
%

We next consider the following system, which is intimately related to mean--field FBSDE as introduced by \cite{carmona2013mean}.
\begin{equation}\label{edsr:agent:max}
\begin{cases} \displaystyle Y_t^{\star}(\chi,\xi)=U_A(\xi)+\int_t^T  g^\star(s,X, Y_s^{\star}(\chi,\xi), Z_s^{\star}(\chi,\xi), \mu,q_s,\chi_s)  ds-\int_t^TZ_s^{\star}(\chi,\xi) \sigma_s(X) dW_s,\\[0.8em]
\displaystyle \mathbb P^{\mu,q,a^\star(\cdot,X,Z^\star_\cdot(\chi,\xi),\mu,q_\cdot)}\circ (X)^{-1}=\mu,\\[0.8em]
 \displaystyle \mathbb P^{\mu,q,a^\star(\cdot,X,Z^\star_\cdot(\chi,\xi),\mu,q_\cdot)}\circ (a^\star(s, X, Z_s^{\star}(\chi,\xi), \mu,q_s))^{-1}=q_s, \text{ for $a.e.$ $s\in[0,T]$},
\end{cases}
\end{equation}
where the process $a^\star(\cdot,X,Z^\star_\cdot(\chi,\xi),\mu,q_\cdot)$ from Assumption \ref{assumption:unique} satisfies
\begin{equation}\label{maximizer:g:eq}g^\star(\cdot,X, Y^{\star}(\chi,\xi), Z^{\star}(\chi,\xi), \mu,q,\chi)=g(\cdot,X, Y^{\star}(\chi,\xi), Z^{\star}(\chi,\xi),a^\star(\cdot,X,Z^\star_\cdot(\chi,\xi),\mu,q_\cdot), \mu,q,\chi).\end{equation}
\noindent Notice that we intrinsically invoke measurable selection arguments (see for instance \cite{karoui2013capacities2} for detailed explanations) to ensure the measurability of $a^\star(\cdot,X,Z^\star_\cdot(\chi,\xi),\mu,q_\cdot)$. This system will provide us the required probabilistic representation of the solution to the mean--field game of the Agent. Before presenting and proving this link, we start by defining a solution to \eqref{edsr:agent:max}. Once more, the constants $\ell$, $m$ and $\underline m$ are the ones appearing in Assumption $(\mathbf{C}^{p,\ell,m,\underline m})$.
\begin{Definition}
\label{def:solution:edsrmuq}
A solution to the mean--field BSDE \eqref{edsr:agent:max} is a quadruplet $(Y^\star,Z^\star,\mu,q)\in\mathbb S_{\rm exp}(\mathbb R) \times \mathbb H_{\rm exp}^{\lambda r}(\mathbb R)\times \mathcal P(\mathcal C)\times \mathfrak P(\mathbb R) $, with $\lambda= (\ell+m)/(\underline m+1-\ell)$, satisfying the system \eqref{edsr:agent:max} for any $t\in[0,T]$, $\P-a.s.$ 
\end{Definition}

Inspired by \cite{elie2016contracting}, we aim at providing an equivalence result between a solution to \textbf{(MFG)}$(\chi,\xi)$ in the set ${\rm MF}^r(\chi,\xi)$ and a solution to BSDE \eqref{edsr:agent:max} in the sense of Definition \ref{def:solution:edsrmuq}. We have the following theorem which provides such a result, together with a characterisation of an optimal effort for the Agent in terms of maximisers of $g$. Notice that such a link is classical and expected, and was already obtained by Carmona and Lacker \cite{carmona2015probabilistic}, but in a case where the Hamiltonian is Lipschitz and thus does not cover our setting.

\begin{Theorem}\label{thm:agent}
Let Assumptions $(\mathbf{B}^{p,\ell,\eta})$, $(\boldsymbol{\sigma})$, $\mathbf{(K)}$ and $(\mathbf{C}^{p,\ell,m,\underline m})$ hold true for some $(p,\ell,m,\underline m,\eta)\in[1,+\infty)\times(0,+\infty)\times[\ell,+\infty)\times(\ell-1,+\infty)\times[0,+\infty)\times (1,+\infty)$. Fix $(\chi,\xi)\in \mathfrak C$.
\begin{itemize}[leftmargin=*]
\item Assume that $(\chi,\xi)\in \Xi$, \textit{i.e.}, the system $\mathbf{(MFG)}(\chi,\xi)$ admits a solution in ${\rm MF}^r(\chi,\xi)$ denoted by $(\mu,q,\alpha^{\star})$. Then there exists a solution $(Y^\star,Z^\star,\mu,q)$ to BSDE \eqref{edsr:agent:max} such that $\alpha^{\star}_t=a^\star(t,X, Z^\star,\mu,q,\chi)$ almost surely. In this case, we have the following decomposition for $U_A(\xi)$
\begin{equation}\label{decomposition:uaxi}U_A(\xi)=Y_0^{\star}-\int_0^T  g^\star(t,X, Y_t^{\star}, Z_t^{\star}, \mu,q_t,\chi_t)  dt+\int_0^TZ_t^{\star} \sigma_t(X) dW_t.\end{equation}
\item Conversely, if there exists a solution to BSDE \eqref{edsr:agent:max} denoted by $(Y^\star(\chi,\xi),Z^\star(\chi,\xi),\mu,q) \in\mathbb S_{\rm exp}(\mathbb R) \times \mathbb H_{\rm exp}^{\lambda r}(\mathbb R)\times \mathcal P(\mathcal C)\times \mathfrak P(\mathbb R)$, then the system $\mathbf{(MFG)}(\chi,\xi)$ admits as a solution in ${\rm MF}^r(\chi,\xi)$, given by the triplet $(\mu,q,a^\star(\cdot,X,Z^\star_\cdot(\chi,\xi),\mu,q_\cdot))$, where $a^\star$ is characterized by \eqref{maximizer:g:eq}.
\end{itemize}
\end{Theorem}

\subsection{A convenient characterisation of $\Xi$}

In order to provide a relevant Hamilton--Jacobi--Bellman equation to solve the Principal's problem \eqref{PrincipalProblem}, and following the general approach to contracting problems initiated by Cvitani\'c, Possama\"i and Touzi \cite{cvitanic2014moral,cvitanic2015dynamic}, we need to have a convenient probabilistic representation of the value function of the Agent, for any contract $(\chi,\xi)\in\Xi$. We will now show in this section that Theorem \ref{thm:agent} is tailor--made for that purpose.

\vspace{0.5em}
Let us start by introducing a convenient notation, and define the set $\mathcal X$ as the set of $\R_+-$valued $\F-$predictable processes $\chi$ such that
$$\mathbb E\left[\exp\left(h\int_0^T\tilde u_A(\chi_s)ds\right)\right]<+\infty,\text{ for any $h\geq 0$.}$$

\vspace{0.5em}
Our first step is to introduce an appropriate system of coupled and controlled McKean--Vlasov type SDEs, which basically amounts to look at the BSDE \eqref{edsr:agent:max} in a forward manner. For any $(Y_0,Z)\in\mathbb R\times  \mathbb H_{\rm exp}^{\lambda r}(\mathbb R)$, with once more $\lambda=(\ell+m)/(\underline m+1-\ell)$, and any $\chi\in\mathcal X$, we introduce for \textcolor{black}{$a^\star$ given by Assumption \ref{assumption:unique}.}
 \begin{equation}\label{systemeEDS}
\begin{cases}
\displaystyle X_t=\psi+\int_0^tb(s,X, \mu, q_s,  a^\star(s,X,Z_s,\mu,q_s))ds+\int_0^t\sigma_s(X) dW^{\mu,q,a^\star(\cdot,X,Z_\cdot,\mu,q_\cdot)}_s,\\[0.8em]
\displaystyle Y_t^{Y_0,Z}(\chi)=Y_0+ \int_0^t\left(b(s,X, \mu, q_s,  a^\star(s,X,Z_s,\mu,q_s))Z_s-g^\star(s,X, Y_s^{Y_0,Z}(\chi), Z_s, \mu,q_s,\chi_s) \right) ds\\
\displaystyle\hspace{5em}+ \int_0^tZ_s \sigma_s(X) dW^{\mu,q,a^\star(\cdot,X,Z_\cdot,\mu,q_\cdot)}_s,\\[0.8em]
\displaystyle \mu=\mathbb P^{\mu,q,a^\star(\cdot,X,Z_\cdot,\mu,q_\cdot)} \circ X^{-1},\\[0.8em]
\displaystyle q_t=\mathbb P^{\mu,q,a^\star(\cdot,X,Z_\cdot,\mu,q_\cdot)} \circ (a^\star(t,X,Z_t,\mu,q_t))^{-1},\ \text{for Lebesgue $a.e.$ $t\in[0,T]$}.
 \end{cases} 
\end{equation}
A solution of this system will be required to satisfy the following properties.
\begin{Definition}\label{definition:solsystemeEDS} A solution of the system \eqref{systemeEDS} is a quadruplet $(X,Y^{Y_0,Z}(\chi),\mu,q)$ satisfying \eqref{systemeEDS} for any $t\in[0,T]$, $\P-a.s.$, such that in addition $Y^{Y_0,Z}(\chi)\in\mathbb S_{\rm exp}(\mathbb R)$. We call $\mathcal Z(\chi)$ the subset of $Z\in \mathbb H_{\rm exp}^{\lambda r}(\mathbb R)$, with $\lambda=(\ell+m)/(\underline m+1-\ell)$, such that there is a solution to \eqref{systemeEDS} for this given $Z$.
\end{Definition}

We now define $\widehat {\Xi}$ as follows
$$\widehat{\Xi}:=\left\{\Big(\chi,U_A^{(-1)}\big(Y_T^{Y_0,Z}(\chi)\big)\Big),\ \chi\in\mathcal X,\ Y_0\geq R_0,\ Z\in\mathcal Z(\chi)\right\}.$$

We have the following characterisation of the set $\Xi$ as an immediate consequence of Theorem \ref{thm:agent}.
\begin{Corollary}\label{corollary:xi}
Let Assumptions $(\mathbf{B}^{p,\ell,\eta})$, $(\boldsymbol{\sigma})$, $\mathbf{(K)}$ and $(\mathbf{C}^{p,\ell,m,\underline m})$ hold true for some $(p,\ell,m,\underline m,\eta)\in[1,+\infty)\times(0,+\infty)\times[\ell,+\infty)\times(\ell-1,+\infty)\times (1,+\infty)$. Then $\Xi=\widehat{\Xi}.$ 
\end{Corollary}

\section{The Principal's problem: optimal control of a McKean-Vlasov SDE}\label{section:mckean}

This section is devoted to the proof of our main result, that is to say that the problem of the Principal amounts to solving a so--called mean--field type control problem.

\subsection{Rewriting the Principal's problem }\label{section:rewritingPP}
Using the characterisation of $\Xi$ provided by Corollary \ref{corollary:xi}, we have immediately\begin{align*}
U_0^P&= \sup_{(\chi,\xi) \in \Xi} \, \sup_{(\mu,q,\alpha) \in \rm{MF}^r(\chi,\xi)} \,  \,\mathbb E\left[U_P\left(X_T-\xi-\int_0^T \chi_s ds\right)\right]
= \sup_{Y_0\geq R_0} U^P_0(Y_0),
\end{align*}
where
 $$U_0^P(Y_0):=\sup_{(\chi,Z)\in \mathcal X\times\mathcal Z(\chi)} \, \sup_{(\mu,q,\alpha) \in {\rm{MF}^r}\big(\chi,U_A^{(-1)}\big(Y_T^{Y_0,Z}\big)\big)}  \,\mathbb E^{\P^{\mu,q,\alpha}}\bigg[\left(X_T-U_A^{-1}(Y^{Y_0,\mathcal Z}_T)-\int_0^T \chi_s ds\right)\bigg].$$

Hence, for any $Y_0\geq R_0$, we can identify $U_0^P(Y_0)$ as the value function of a stochastic optimal control problem with a two--dimensional state variable $M^{\chi,Z}:=(X,Y^{Y_0,Z}(\chi))^\top$ controlled by the processes $(\chi,Z)\in\mathcal X\times\mathcal Z(\chi)$. Introducing the following two functions, depending on the map $a^\star$ from Assumption \ref{assumption:unique} 

\vspace{0.5em}
$\bullet$ $C:[0,T]\times \mathcal C^2 \times \mathcal P(\mathcal C^2)\times \mathfrak P(\mathbb R) \times \mathbb R\times \R_+\longrightarrow \mathbb R^2$ defined for any $(t,m,\mu,q,z,\chi)\in [0,T]\times \mathcal C^2\times \mathcal P(\mathcal C^2)\times \mathfrak P(\mathbb R)\times\mathbb R\times\mathbb R_+ $ by
$$C(t,m,\mu,q,z,\chi):= \begin{pmatrix}
    b(t,m^1, \mu^1, q, a^\star(t,m^1,z,\mu^1,q)) \\
k(t,m^1,\mu^1,q)m^2(t) + c(t,m^1,\mu^1,q,a^\star(t,m^1,z,\mu^1,q))-u_A(t,m^1,\mu^1,q,\chi)
\end{pmatrix}. $$

$\bullet$  $S: [0,T]\times \mathcal C^2\times \mathbb R \longrightarrow \mathcal M_{2}(\mathbb R)$ defined for any $(t,m,z)\in [0,T]\times \mathcal C^2\times \mathbb R$ by
$$S(t,m,z):= \begin{pmatrix}
  \sigma_t(m^1)& 0\\
  \sigma_t(m^1) z&0
\end{pmatrix},$$
the dynamics of $M^{\chi,Z}$ is then given by\footnote{In full generality, we should have let the initial value of $Y$ be a probability measure as well. However, since the law of $Y$ never appears in the problem, we fixed it as a constant for simplicity. Notice that in the examples below, we will modify the criterion of the Principal and let the law of $Y$ play a role. This would necessitate to rewrite the current section, which we refrain from doing for ease of notations.}
\begin{equation}\label{sde:M}
\begin{cases}
\displaystyle M_t^{\chi,Z}= \begin{pmatrix}\psi\\Y_0\end{pmatrix}+\int_0^tC(s,M^{\chi,Z},\mu, q_s, Z_s,\chi_s)ds+ \int_0^tS(s,M^{\chi,Z},Z_s){\mathbf 1}_{2}dW^{a^\star(M^1,Z,\mu^1,q)}_s,\\[0.8em]
 \displaystyle \mu=\mathbb P^{a^\star(M^1,Z,\mu^1,q)}\circ \left(M^{\chi,Z}\right)^{-1},\\[0.8em]
\displaystyle q_t=\mathbb P^{a^\star(M^1,Z,\mu^1,q)}\circ \left(a^\star(t,M^1,Z_t,\mu^1,q_t)\right)^{-1},\ \text{for Lebesgue $a.e.$ $t\in[0,T]$}.
\end{cases}
\end{equation}

Introducing finally the map $G:\mathbb R^2 \longrightarrow \mathbb R$, such that for any $m\in \mathbb R^2$, $G(m):=m^1-U_A^{-1}(m^2), $
the problem of the Principal can finally be linked to
\begin{equation}\label{pb:mckeanvlasov}
 U_0^P(Y_0)= \sup_{(\chi,Z)\in \mathcal X\times\mathcal Z(\chi)} \, \sup_{(\mu,q,\alpha) \in {\rm{MF}^r}\big(\chi,U_A^{(-1)}\big(Y_T^{Y_0,Z}\big)\big)}  \,\mathbb E^{\P^{\mu,q,\alpha}}\bigg[U_P\bigg(G\big(M_T^{\chi,Z}\big)-\int_0^T\chi_s ds\bigg)\bigg],
\end{equation}

\begin{Remark}
Notice that the dynamics of $M$ does not \textcolor{black}{involve the second marginal of $\mu$. } We nonetheless used this notation to stay within the framework of {\rm \cite{pham2015bellman}} or {\rm \cite{carmona2015forward}}.
\end{Remark}
Solving such a problem in full generality goes far beyond the scope of the present paper, and this question is actually the subject of a great number of current studies. Once again, our main message here is that the {\it a priori} quite complicated problem faced by the Principal, consisting in finding a Stackelberg equilibrium between himself and a mean--field equilibrium of interacting Agents, is actually amenable to a dynamic programming approach, and leads to a simpler problem of mean--field type control. 

\vspace{0.5em}
In the subsequent section, we will describe informally one possible approache to solve the problem of the Principal which has been proposed in the literature, before providing several explicitly solvable examples in Section \ref{section:ex}.
\vspace{0.5em}

\subsection{An approach using the dynamic programming principle in the Markovian case}\label{section:PW}

The approach that we present here (but not the first chronologically) to solve \eqref{pb:mckeanvlasov} is mainly based on the recent papers \cite{bayraktar2016randomized,pham2016linear,pham2015bellman,pham2015discrete,pham2016dynamic} (see the references therein for earlier results) and consists in using the dynamic programming principle and solving the corresponding Hamilton--Jacobi--Bellman equation in an infinite dimensional space. Of course, this requires to work in a Markovian framework, namely that, abusing notations slightly, for any $(t,x,\mu,q,a)\in[0,T]\times\mathcal C\times \mathcal P(\mathcal C)\times\mathfrak P(\R)\times \R_+$
$$b(t,x,\mu,q,a)=b(t,x(t),\mu_t,q,a),\ c(t,x,\mu,q,a)=c(t,x(t),\mu_t,q,a),\ k(t,x,\mu,q)=k(t,x(t),\mu_t,q),$$
where for any $(t,\mu)\in[0,T]\times\mathcal P(\mathcal C)$, we define $\mu(t)\in\mathcal P(\mathbb R)$ by
$$\mu_t[A]:=\mu\left[\left\{\omega\in\Omega,\ \omega(t)\in A\right\}\right],\ \text{for every $A\in\mathcal B(\mathbb R)$}.$$

In order to give a precise meaning to the notion of differentiability in this framework, we follow \cite[Section 4.2]{pham2015bellman}, which is based on the original ideas of Lions \cite{lions2007theorie} (see also the lecture notes of Cardaliaguet \cite{cardaliaguet2010notes}), and introduce the differentiability with respect to a probability measure, based on a lifting procedure. 

\vspace{0.5em}
Let $n$ be a positive integer and let $u: \mathcal P(\mathbb R^n) \longrightarrow \mathbb R$. We define $\tilde u: L^2(\Omega,\mathcal F_T,\mathbb R^n)\longrightarrow \mathbb R$ by $\tilde u(\eta):= u(\mathcal L(\eta))$ for any $\eta \in L^2(\Omega,\mathcal F_T,\mathbb R^n)$. We then say that $u$ is differentiable on $\mathcal P(\mathbb R^n)$ if $\tilde u$ is Fr\'echet differentiable on $L^2(\Omega,\mathcal F_T,\mathbb R^n)$ and we denote by $[D\tilde u](\eta)$ its Fr\'echet derivative in the direction of $\eta$, which can be identified as a linear operator from $L^2(\Omega,\mathcal F_T,\mathbb R^n)$ into $\mathbb R$. According to Riesz's Theorem, for any $\eta \in L^2(\Omega,\mathcal F_T,\mathbb R^n)$ there exists $D\tilde u(\eta)\in L^2(\Omega,\mathcal F_T,\mathbb R^n)$ such that for any $Y\in L^2(\Omega,\mathcal F_T,\mathbb R^n)$
$$[D\tilde u](\eta)(Y)=\mathbb E[D\tilde u(\eta) \cdot Y]. $$
We then introduce the notation $\partial_\rho u(\mathcal L(\eta)):\mathbb R^n \longrightarrow \mathbb R^n$ as follows, $ D\tilde u (\eta)=: \partial_\rho u(\mathcal L(\eta))(\eta),$
and call the latter the derivative of $u$ at $\rho=\mathcal L (\eta)$. We now say that $u$ is partially twice differentiable on $\mathcal P(\mathbb R^n)$ if
$u$ is differentiable in the above sense, and such that the mapping $(\rho,x)\in \mathcal P(\mathbb R^n)\times \R^n \longmapsto \partial_\rho u(\rho)(x)$ is continuous at any point $(\rho,x)$, with $x$ being in the support of $\rho$, and if for any $\rho\in \mathcal P(\R^n)$, the map $x\in \R^n\longmapsto \partial_\rho u(\rho)(x)$ is differentiable. We denote the gradient of $\partial_\rho u(\rho)(x)$ by $\partial_x\partial_\rho u(\rho)(x)\in \mathbb R^n\times \R^n$. We also recall that a generalised It\=o's formula has been proved by Chassagneux, Crisan and Delarue \cite{chassagneux2014probabilistic}, for functions from $\mathcal P(\mathbb R^n)$ into $\R$. We will need to distinguish between two cases, depending on whether $C$ depends on the variable $q$ or not.
\subsubsection{No dependance on the law of the controls}\label{section:pasdependent}
In this section, we assume that the drift $b$, the cost function $c$ and the discount factor $k$ do not depend on the variable $q$, which means that the drift of the Agent's output is only impacted by the actions of the other players through their outputs, and not their actions. This is the situation considered in \cite{pham2016dynamic}, from which we deduce that the Hamilton--Jacobi--Bellman equation associated with the McKean--Vlasov optimal control problem \eqref{pb:mckeanvlasov} is
\begin{equation}\label{PDE:PW}
\begin{cases}
\displaystyle 
       -\partial_t v(t,\rho)-\underset{{(\chi,z)\in \mathbb R^+\times\mathbb R}}{\sup}  \, H(t,\rho,\partial_\rho v, \partial_x \partial_\rho,z,\chi )=0,\ (t,\rho)\in [0,T)\times \mathcal P(\mathbb R^2),\\
     \displaystyle   v(T,\rho)=\int_{\mathbb R^2} U_P\left(G(x)\right)\rho(dx),\ \rho\in\mathcal P(\R^2),
      \end{cases}
\end{equation}
with
$$H(t,\rho,\partial_\rho v, \partial_x \partial_\rho v,z,\chi ):= \int_{\mathbb R^2} \mathbb L^{\chi,z} v(t,\rho)(x) \rho(dx)$$ and
where for any $\varphi:[0,T]\times \mathcal P(\mathbb R^2)\longrightarrow \R$, which is continuously differentiable in $t$ and twice continuously differentiable in $\rho$, for any $(t,\rho)\in[0,T]\times \mathcal P(\mathbb R^2)$, $(\chi,z)\in \R^+\times\mathbb R) $, the map $ \mathbb L^{\chi,z} \varphi(t,\rho)$ is defined from $\mathbb R^2$ into $\mathbb R$ by
\begin{align*}
 \mathbb L^{\chi,z} \varphi(t,\rho)(x)&:=\partial_\rho \varphi (t,\rho) \cdot C(t, x, \rho,z,\chi)+\frac12 \text{Tr}\left[\partial_x\partial_\rho \varphi(t,\rho)(x) SS^\top(t,x,z)\right]-\chi, \ x\in\R^2.
\end{align*}
 \textcolor{black}{Since in Section \ref{section:ex} we will focus on examples fitting with the framework of Section \ref{section:dependent} below, the proof of the following verification result to solve \eqref{pb:mckeanvlasov} is omitted. It is however a classical result of stochastic control theory on the space of measure, and we refer to Theorem \ref{thm:verif:depend} in the section below for a version in a more general framework, but with a more restrictive class of controls.  }

 \begin{Theorem}\label{thm:verif:nodepend} Let $v$ be a continuous map from $[0,T]\times \mathcal P_2(\mathbb R^2)$ into $\mathbb R$ such that $v(t,\cdot)$ is twice continuously differentiable on $\mathcal P_2(\mathbb R^2)$ and such that $v(\cdot,\mu)$ is continuously differentiable on $[0,T]$. Suppose that $v$ is solution to \eqref{PDE:PW} such that the supremum is attained for some optimizers of $H$ denoted by $\chi^\star(t,\rho,\partial_\rho v,\partial_x\partial_\rho v)$ and $Z^\star(t,\rho,\partial_\rho v,\partial_x\partial_\rho v)$ for any $(t,\rho)\in [0,T]\times \mathcal P_2(\mathbb \R^2)$ such that $(t,\rho)\longmapsto \chi^\star(t,\rho,\partial_\rho v,\partial_x\partial_\rho v), Z^\star(t,\rho,\partial_\rho v,\partial_x\partial_\rho v)$ are measurable. \vspace{0.3em}
 
 Let $M_0$ be a square random variable $\mathbb R^2$ valued with law $\mu_0\in\mathcal P_2(\R^2)$ and assume moreover that the following McKean--Vlasov SDE \begin{equation}\label{sde:M:optim}
\begin{cases}
\displaystyle M_t^{\chi^\star,Z^\star}= \begin{pmatrix}\psi\\Y_0\end{pmatrix}+\int_0^tC(s,M^{\chi^\star,Z^\star},\mu, Z^\star(s,\mu,\partial_\mu v,\partial_x\partial_\mu v),\chi^\star(s,\mu,\partial_\mu v,\partial_x\partial_\mu v))ds\\[0.8em]
\displaystyle \hspace{4.5em}+ \int_0^tS(s,M^{\chi^\star,Z^\star},Z^\star(s,\mu,\partial_\mu v,\partial_x\partial_\mu v)){\mathbf 1}_{2}dW^{a^\star(M^1,Z,\mu^1)}_s,\\[0.8em]
 \displaystyle \mu=\mathbb P^{a^\star(M^1,Z^\star,\mu^1)}\circ \left(M^{\chi^\star,Z^\star}\right)^{-1}\\

\end{cases}
\end{equation}
admits a solution $M^{\chi^\star,Z^\star}$. Then if the pair of control $\left(Z^\star(s,\mu,\partial_\mu v,\partial_x\partial_\mu v),\chi^\star(s,\mu,\partial_\mu v,\partial_x\partial_\mu v)\right)$ is in $\mathcal X\times \mathcal Z(\chi)$, we have
$$v(0,\mu_0 )= U_0(Y_0), $$
and $(\chi^\star(\cdot,\mu,\partial_\mu v,\partial_x\partial_\mu v), Z^\star(\cdot,\mu,\partial_\mu v,\partial_x\partial_\mu v))$ are optimal in the problem of the Principal \eqref{pb:mckeanvlasov}.
 \end{Theorem}
 
%
%
 
 \begin{Remark}
 \textcolor{black}{The function $v$ solution to PDE \eqref{PDE:PW} has to be seen as the continuation utility of the Principal. The proof of this theorem requires only to consider $v$ at time $t=0$, but more generally, a dynamic version of it is of course available in the spirit of {\rm \cite{pham2016dynamic}} for instance. }
 \end{Remark}

\subsubsection{Dependance on the law of the controls}\label{section:dependent}
We now turn to the case where the dependance with respect to the law $q$ impacts the drift $b$, the cost function $c$ and the discount factor $k$. This situation was recently studied in \cite{pham2015bellman}, and requires some modifications compared to the previous case. 

\vspace{0.5em}
Let $p\geq 1$ and let $\text{Lip}_p([0,T]\times \R\times \mathcal P(\mathbb R); \R)$ be the set of deterministic measurable functions $\tilde Z:[0,T]\times \R\times\mathcal P(\R)\longrightarrow\R$, such that for any $t\in [0,T]$, the map $(x,\mu) \longmapsto \tilde Z(t,x,\mu)$ is such that there exists a constant $C>0$ satisfying for any $(t,x,x',\mu,\mu')\in [0,T]\times \R^2\times \R^2\times \mathcal P(\R^2)\times \mathcal P(\R^2)$
$$|\tilde Z(t,x,\mu)-\tilde Z(t,x',\mu')|\leq C\left(\|x-x'\|+\mathcal W_{\R^2,p}(\mu,\mu') \right) ,
\text{ and }\int_0^T \tilde Z(t,\mathbf 0_2,\delta_{\mathbf 0_2}) dt<+\infty. $$  We now restrict our attention to control processes $(\chi,Z)\in\mathcal X\times\mathcal Z(\chi)$ which are actually (partially) Markovian, in the sense that there is a map $\tilde Z:[0,T]\times\mathbb R^2\times\mathcal P(\mathbb R^2)$ in $\text{Lip}_p([0,T]\times \R\times \mathcal P(\mathbb R); \mathbb R)$ for some $p\geq 1$, such that
$$Z_t=\tilde Z\big(t,M^{\chi,\tilde Z}_t,\mathcal L(M^{\chi,\tilde Z}_t)\big),\ t\in[0,T],$$ and such that
SDE \eqref{sde:M} admits a solution $M^{\chi,Z}$. We denote by $\tilde {\mathcal Z}_p$ the set of such processes.
The main difference with Section \ref{section:pasdependent} lies in the HJB equation associated with the Principal's problem \eqref{pb:mckeanvlasov} which becomes
\begin{equation}\label{PDE:PW:dependancyq}
\begin{cases}
\displaystyle 
       -\partial_t v(t,\rho)-\underset{{(\chi,\tilde z)\in \mathbb R^+\times \text{Lip}(\mathbb R^2;\mathbb R)}}{\sup}  \, \, \tilde H(t,\rho,\partial_\rho v, \partial_x \partial_\rho,\tilde z,\chi ) =0,\ (t,\rho)\in [0,T)\times \mathcal P(\mathbb R^2),\\
     \displaystyle   v(T,\rho)=U_P\left(\int_{\mathbb R^2} G(x)\rho(dx) \right),\ \rho\in\mathcal P(\R^2),
      \end{cases}
\end{equation}
where $\text{Lip}(\mathbb R^2;\mathbb R)$ denotes the set of Lipschitz function from $\mathbb R$ into $\mathbb R$ with 
$$\tilde H(t,\rho,\partial_\rho v, \partial_x \partial_\rho,\tilde z,\chi ):= \int_{\mathbb R^2} \widetilde{\mathbb L}^{\chi, \tilde z} v(t,\rho) \rho(dx)$$  and
where for any $\varphi:[0,T]\times \mathcal P(\mathbb R^2)\longrightarrow \R$, which is continuously differentiable in $t$ and twice continuously differentiable in $\rho$, for any $(t,\rho)\in[0,T]\times \mathcal P(\mathbb R^2)$, $(\chi,\tilde z)\in \R^+\times\text{Lip}(\mathbb R^2;\mathbb R) $, the map $ \widetilde{\mathbb L}^{\chi,\tilde z} \varphi(t,\rho)$ is defined from $\mathbb R^2$ into $\mathbb R$ by
\begin{align*}
 \widetilde{\mathbb L}^{\chi,\tilde z} \varphi(t,\rho)(x)&:=\partial_\rho \varphi (t,\rho) \cdot C(t, x, \rho,q_t,\tilde z(x),\chi)+\frac12 \text{Tr}\left[\partial_x\partial_\rho \varphi(t,\rho)(x) SS^\top(t,x,\tilde z(x))\right]-\chi, \ x\in\R^2.
\end{align*}

 \textcolor{black}{We thus have the following verification result to solve the problem \eqref{pb:mckeanvlasov}  }

 \begin{Theorem}\label{thm:verif:depend} Let $v$ be a continuous map from $[0,T]\times \mathcal P_2(\mathbb R^2)$ into $\mathbb R$ such that $v(t,\cdot)$ is twice continuously differentiable on $\mathcal P_2(\mathbb R^2)$ and such that $v(\cdot,\mu)$ is continuously differentiable on $[0,T]$. Suppose that $v$ is solution to \eqref{PDE:PW} such that the supremum is attained for some optimizers of $\tilde H$ denoted by $\chi^\star(t,\rho,\partial_\rho v,\partial_x\partial_\rho v)$ and $\widetilde{Z}^\star(t,\cdot,\rho)\in \text{Lip}(\mathbb R^2;\mathbb R)$ for any $(t,\rho)\in [0,T]\times \mathcal P_2(\mathbb \R^2)$ such that $(t,\rho)\longmapsto \chi^\star(t,\rho,\partial_\rho v,\partial_x\partial_\rho v)$ is measurable and $(t,x,\rho)\longmapsto \widetilde{Z}^\star(t,x,\rho)$ is in $\text{Lip}_p([0,T]\times \R\times \mathcal P(\mathbb R); \mathbb R)$ for some $p\geq 1$. \vspace{0.3em}
 
 Let $M_0$ be a square random variable $\mathbb R^2$ valued with law $\mu_0\in\mathcal P_2(\R^2)$ and assume moreover that the following McKean-Vlasov SDE \begin{equation}\label{sde:M:optim2}
\begin{cases}
\displaystyle M_t^{\chi^\star,Z^\star}= \begin{pmatrix}\psi\\Y_0\end{pmatrix}+\int_0^tC(s,M^{\chi^\star,\widetilde{Z}^\star},\mu, \widetilde{Z}^\star(s,M_s^{\chi^\star,Z^\star},\mu),\chi^\star(s,\mu,\partial_\mu v,\partial_x\partial_\mu v))ds\\[0.8em]
\displaystyle \hspace{4.5em}+ \int_0^tS(s,M^{\chi^\star,\widetilde{Z}^\star}, \widetilde{Z}^\star(s,M_s^{\chi^\star,Z^\star},\mu)){\mathbf 1}_{2}dW^{a^\star(M^1,\widetilde{Z}^\star,\mu^1)}_s,\\[0.8em]
 \displaystyle \mu=\mathbb P^{a^\star(M^1, \widetilde{Z}^\star(\cdot,M^{\chi^\star,Z^\star},\mu),\mu^1)}\circ \left(M^{\chi^\star,\widetilde{Z}^\star}\right)^{-1},\\
\displaystyle q_t=\mathbb P^{a^\star(M^1,\tilde Z^\star,\mu^1,q)}\circ \left(a^\star(t,M^1,\widetilde{Z}^\star(\cdot,M^{\chi^\star,Z^\star},\mu),\mu^1,q_t)\right)^{-1},\ \text{for Lebesgue $a.e.$ $t\in[0,T]$}.
\end{cases}
\end{equation}
admits a solution $M^{\chi^\star,Z^\star}$. Then if the pair of control $\left(\chi^\star(s,\mu,\partial_\mu v,\partial_x\partial_\mu v), \widetilde{Z}^\star(\cdot,M^{\chi^\star,Z^\star},\mu)\right)$ is in $\mathcal X\times \mathcal Z(\chi)$, we have
$$v(0,\mu_0 )= U_0(Y_0), $$
and $(\chi^\star(\cdot,\mu,\partial_\mu v,\partial_x\partial_\mu v), Z^\star(\cdot,\mu,\partial_\mu v,\partial_x\partial_\mu v))$ are optimal in the problem of the Principal \eqref{pb:mckeanvlasov}.
 \end{Theorem}
 \begin{proof}
 Let $(\chi,Z)$ be any control in $\tilde {\mathcal Z}_p$ with associated $\tilde Z\in \text{Lip}_p([0,T]\times \R\times \mathcal P(\mathbb R); \mathbb R)$. We fix a random variable $M_0$ together with its law $\mu_0$. Let $M^{\chi,Z}$ be a solution to the McKean--Vlasov SDE \eqref{sde:M} with initial condition $(M_0,\mu_0)$. For this control $(\chi,Z)$, we fix some triplet $(\mu,q_t,\alpha)$ in $ {\rm{MF}^r}\big(\chi,U_A^{(-1)}\big(Y_T^{Y_0,Z}\big)\big) $. According to It\=o's formula on the space of measures (see for instance \cite{chassagneux2014probabilistic}), we have classically
\begin{align*}
 v(T,\mu) &=  v(0,\mu_0)+\int_0^T \left(\partial_t v(s,\mu) + \int_\mathbb R \widetilde{\mathbb L}_t^{\chi,\tilde Z} v(s,\mu)(x) \nu(dx)\right) ds\\
 &\leq v(0,\mu_0)+\int_0^T \left(\partial_t v(s,\mu) + \int_\mathbb R \underset{{(\chi,\tilde z)\in \mathbb R^+\times \text{Lip}(\mathbb R^2;\mathbb R)}}{\sup}\, \widetilde{\mathbb L}_t^{\chi,\tilde z} v(s,\mu)(x) \nu(dx)\right) ds\leq  v(0,\mu_0).
 \end{align*}
 Then, we deduce that
$$ v(0,\mu_0)\geq v(T,\mu)= \mathbb E^{\P^{\mu,q,a^\star}}\bigg[\left(X_T-U_A^{-1}(Y^{Y_0, Z}_T)-\int_0^T \chi_s ds\right)\bigg] .$$
Thus, 
$$ v(0,\mu_0)\geq \sup_{(\chi,Z)\in \tilde {\mathcal Z}_p} {\sup_{(\mu,q,\alpha) \in {\rm{MF}^r}\big(\chi,U_A^{(-1)}\big(Y_T^{Y_0,Z}\big)\big)}  } \mathbb E^{\P^{\mu,\alpha}}\bigg[\left(X_T-U_A^{-1}(Y^{Y_0, Z}_T)-\int_0^T \chi_s ds\right)\bigg] .$$

Let now $(\chi^\star, Z^\star)$ optimizer of the Hamiltonian $\tilde H$ with associated measures $(\mu,q_t)\in \mathcal P_2(\R^2)\times \mathfrak P(\mathbb R)$ such that \eqref{sde:M:optim2} admits a solution. Then, all the inequalities above becomes equalities which provides the result.
 \end{proof}

In the example studied in Section \ref{section:ex}, we will provide an intuitive optimal contract such that the corresponding value function of the Principal is smooth and solves directly PDE \eqref{PDE:PW}. Combined with a verification theorem (see for instance Theorems \ref{thm:verif:nodepend} and \ref{thm:verif:depend} in \cite{pham2015bellman} or Theorem \ref{thm:verif:nodepend} above), this ensures that we have indeed solved the Principal's problem. Notice however that in general, one should not expect existence of classical solutions to the above PDE, and one would have to rely on an appropriately defined notion of viscosity solutions. We refer the reader to \cite{pham2016dynamic} for one possible approach.

\subsubsection{On the admissibility of the optimal contract.}\label{section:admissibility}
This short section adresses the problem of admissibility of the optimal contracts derived from the dynamic programming approach we just described. Assume that the HJB equation \eqref{PDE:PW:dependancyq} admits\footnote{We have to make clear the definition of a solution to an HJB equation. Most of the time, we have to deal with solutions in the sense of viscosity for PDEs on Hilbert spaces, introduced by Lions in \cite{lions1988viscosity,lions1989viscosity,lions1989viscosity2}, but in the example studied in Section \ref{section:ex}, we will see that the HJB equation under interest will admits a smooth solution.} a solution, and denote by $(\chi^\star,Z^\star)$ the optimisers of the Hamiltonian. Recall that a contract $(\chi,\xi)$ is admissible, \textit{i.e.} $(\chi,\xi)\in \Xi$ if it is in $\mathfrak C$ and if $\textbf{(MFG)}(\chi,\xi)$ has a solution in $\text{MFI}^r(\chi,\xi)$. According to Corollary \ref{corollary:xi}, the admissibility of $\xi$, and so the existence of a mean--field equilibrium, relies on the existence of a solution in the sense of Definition \eqref{definition:solsystemeEDS}, of the system \eqref{systemeEDS}. As explained in \cite[Remark 7.3]{carmona2015probabilistic}, the main difficulty is that the process $Z^\star$ can depend on the law of $X$ and the existence of a solution to the system \eqref{systemeEDS} is not clear. However, if for instance the process $Z^\star$ is deterministic and belongs to the space $L^{\lambda r}([0,T])$, then the system \eqref{systemeEDS} {\color{black}will much more easily admit a solution in the sense of Definition \ref{definition:solsystemeEDS} since the only dependence on $X$ and its law will then come from $b$ and $\sigma$, and the system also decouples. If existence can be proved, we get as a consequence of Corollary \ref{corollary:xi}, that the contract $\xi^\star=U_A^{(-1)}\big(Y_T^{R_0,Z^\star}(\chi^\star)\big)$ is optimal. This will be exactly the situation encountered in our solvable examples, so that admissibility will come almost for free in these cases, using notably the results from \cite{chiang1994mckean,mishura2016existence}.}

\begin{Remark}
Although we will apply the results developed in {\rm \cite{bayraktar2016randomized,pham2016linear,pham2015bellman,pham2015discrete,pham2016dynamic}} in the examples studied in Section \ref{section:ex}, we also would like to mention a second approach to deal with the Principal problem \eqref{pb:mckeanvlasov} based on {\rm \cite{carmona2015forward}}, consisting in applying the Pontryagin stochastic maximum principle, and then study the associated system of fully--coupled FBSDEs. 
%
%
%
\end{Remark}

\section{Some explicitly solvable examples}\label{section:ex}
In this section, we study an explicit model where the dynamic of the output process depends on its mean and its variance. We fix some $p\geq 1$ associated to the definition of the space $\tilde {\mathcal Z}_p$.
 \subsection{Quadratic dynamics and power cost}\label{example:meanvar}
Let $\alpha\in [0,\frac12)$ and $\beta_1,\beta_2, \gamma\geq 0$. We define for any $(s,x,\mu,q,a)\in [0,T]\times \mathbb R\times  \mathcal P(\mathcal C) \times \mathfrak P(\mathbb R)\times \mathbb R_+,$
 $$ b(s,x,\mu,q,a):=a+\alpha x +\beta_1 \int_\R z d\mu_s(z)+\beta_2 \int_\R z dq_s(z) -\gamma V_\mu(s), $$
$$V_\mu(s):=  \int_\R |z|^2 d\mu_s(z) - \left|\int_\R z d\mu_s(z) \right|^2, \text{ and for fixed $n>1$, } c(s,x,\mu,q,a):=c\frac{|a|^n}{n},\; c>0.$$
Recalling the definitions of $\ell,m,\underline m$, we check immediately that Assumptions $(\mathbf{B}^{p,\ell,\eta})$ and $(\mathbf{C}^{p,\ell,m,\underline m})$ hold with $\ell=1$, $m=\underline m=n-1$.
 
 \vspace{0.5em}
 In this model, we have a ripple effect in the sense that both the value of the project managed by the representative, and the average values of the projects managed by the other, impact linearly through the parameter $\beta_1$ the drift of $X$. The better the situation is (that it the further away from $0$ $X$ and his mean are), the better it will be, and conversely if things start to go wrong. Similarly, we assume that large disparities between the values of the projects have a negative impact on the future evolution of the value of the firm, through the variance penalisation parameter $\beta_2$. We interpret this as an indicator of possible frailness and potential instability of the firm, that should be avoided.
 
 \vspace{0.5em}
For the sake of simplicity, we assume in this applicative section that $u_A=0$ (the Agent only receives a terminal payment $\xi$), $k=0$ (no discount factor), $U_A(x)=x, \; x\in \R$ (the Agent is risk neutral). To alleviate notations, we omit the dependence with respect to $\chi$ in every objects previously defined and we denote $\kappa:= \alpha+\beta_1$.\vspace{0.5em}

\noindent  \textcolor{black}{We also assume that the Principal is risk neutral, which can be at least informally justified by appealing to propagation of chaos arguments. Formally, we considers a $N-$players model with exponential utility for the Principal and we assume furthermore\footnote{We would like to thank one of the referee who has suggested this interpretation of the factor $\frac1N$ in \eqref{eq:pbprincipalN}.} that the Principal assignes projects to the Agents so that each individual output process is scaled by the total number of Agents $N$ to ensure stability as $N$ grows of the cumulative output. The problem of the Principal thus writes
\begin{equation}\label{eq:pbprincipalN}U_0^{P,N}:=\sup_{(\chi,\xi)\in\Xi^N} \mathbb E^{\mathbb P_N^{\alpha}}\left[-\exp\left(-\frac{R_P}{N}\left(X_T-\xi-\int_0^T\chi_sds\right)\cdot {\bf 1}_N\right)\right],\end{equation}
where $\Xi^N$ is the set of contracts proposed to the $N$ Agents satisfying technical conditions (see \cite{elie2016contracting,mastrolia2017moral} for further details). By taking the limit when $N\longrightarrow \infty$ in the Principal's problem \eqref{eq:pbprincipalN}, that we expect the Principal to become risk--neutral. In other words, the mean--field version of \eqref{eq:pbprincipalN} for the Principal is
\begin{equation}\label{PrincipalProblem}
U_0^P:= \sup_{(\chi,\xi)\in \Xi}\, \sup_{(\mu,q,\alpha)\in {\rm MF}^r(\chi,\xi)}\, \mathbb E^{\P^{\mu,q,\alpha}}\left[X_T-\xi-\int_0^T \chi_s ds\right].
\end{equation}
Another interpretation\footnote{We thank a referee for proposing this explanation to us.} of the risk-neutrality of the Principal is that it is due to the fact that we expect the random average penalised output in the $N$-players' game to converge to a deterministic quantity, equal to the expectation of the penalised output of a representative agent, as $N$ goes to infinity. Thus, the quantity inside the utility of the Principal becomes deterministic, and he simply needs to maximise this deterministic quantity.
}

\vspace{0.5em}
Notice that this example that we consider does not fall into the class of so--called linear--quadratic problems, which are, as far as we know, the only explicitly solved ones in the mean--field game literature. Indeed, the drift of the the diffusion $X$ in our case is quadratic, due to the variance penalisation, while the cost of effort is a power function with any exponent $n>1$.
  \subsubsection{Agent's optimal effort}
Recall now the BSDE associated to the Agent value function $U_t^A(\xi,\mu,q)$ is given by
\begin{align*}
 Y_t^{ \mu,q}(\xi)=&\ \xi+\int_t^T g^\star\big(s,X, Z_s^{\mu,q}(\xi), \mu,q_s\big)  ds-\int_t^T Z_s^{ \mu,q}(\xi) \sigma dW_s,
\end{align*}
where in this context we can compute that
$$ g^\star\big(s,x, z, \mu,q_s\big)=\frac{|z|^{\frac{n}{n-1}}}{c^{\frac1{n-1}}}\left(1-\frac1{n}\right)+z \left(\alpha x+\beta_1\int_\R z d\mu_s(z)+\beta_2 \int_\R z dq_s(z) - \gamma V_\mu(s)\right).$$
Let $a^\star(Z^{\mu,q})$ be the corresponding optimiser, \textit{i.e.}
$$a^\star_s(Z^{\mu,q}):=\left(\frac{\left|Z_s^{\mu,q}\right|}{c}\right)^{\frac{1}{n-1}}. $$ Thus, according to Theorem \ref{thm:agent},  as soon as the corresponding system \eqref{edsr:agent:max} has a solution denoted by $(Y^\star, Z^\star, \mu^\star,q^\star)$, $(\mu^\star,q^\star,a^\star_s)$ is solution to $\textbf{(MFG)}(\xi)$ and $a^\star_s:= a^\star_s(Z^\star)$ is optimal for the Agent.

\subsubsection{Principal's problem}

Recall from Section \ref{section:dependent} that the HJB equation associated with the Principal's problem is
\begin{equation}\label{HJB:exemple}
\begin{cases}
\displaystyle 
       -\partial_t v(t,\rho)-\underset{{(\chi,Z)\in \mathcal X\times \tilde {\mathcal Z}_p}}{\sup}  \,\left(\int_{\mathbb R^2} \mathbb L^{\chi, Z} v(t,\rho) \rho(dx)\right)=0,\ (t,\rho)\in [0,T)\times \mathcal P(\mathbb R^2),\\
     \displaystyle   v(T,\rho)=\int_{\mathbb R^2} G(x)\rho(dx),\ \rho\in\mathcal P(\R^2).
      \end{cases}
\end{equation}

In our particular case, we can solve explicitly HJB equation \eqref{HJB:exemple} and using a verification result we can solve completely the problem of the Principal \eqref{pb:mckeanvlasov}. This lead us to our main result
\begin{Theorem}\label{thm:optimal:ex} The optimal contract for the problem of the Principal is
\begin{align*}
\xi^\star:=&\ 
\delta +\beta_1(1+\beta_2) \int_0^T e^{(\alpha+\beta_1)(T-t)} X_t dt+(1+\beta_2)\left(X_T-e^{(\alpha+\beta_1) T}X_0\right)
\end{align*}
for some constant $\delta$ explicitly given by \eqref{delta_optimal} in the proof. Besides, the contract $\xi^*$ is a Normal random variable and 
 $$
 \xi^* \; \sim \; \Nc\left( R_0 +  \frac{(1+\beta_2)^\frac{n}{n-1}}{n c} \left(e^{\frac{n}{n-1}(\alpha+\beta_1)(T)} -1\right)  \;;\;  \sigma^2 (1+\beta^2)^2 \left(e^{2(\alpha+\beta_1)T}-1\right) \right) \;.
 $$
The associated optimal effort of the Agent is deterministic and given by
$$a^\star_u:=(1+\beta_2)^{\frac{1}{n-1}}\left( \frac{e^{(\alpha+\beta_1) (T-u)}}{c}\right)^{\frac{1}{n-1}},\, u\in [0,T],$$
{\color{black}and the value function of the Principal (at least when $\alpha>0$ and $\alpha\neq\beta_!$, see the proof for the other cases) is
\begin{align*}
v(t,\nu)=&\ \frac{(1+\beta_2)^{\frac{n}{n-1}}}{\kappa c^\frac{1}{n-1}} \left(1-\frac1n \right)^2\big(e^{\kappa\frac{n}{n-1}(T-t)}-1\big)+ \int_\mathbb R x \nu^1(dx) e^{\kappa(T-t)}-\int_\mathbb R x \nu^2(dx)\\
&-\frac{\gamma}{2\alpha-\kappa}\big(e^{2\alpha (T-t)}-e^{\kappa (T-t)}\big) \left(\int_\R |x|^2 \nu(dx)-\left(\int_\R x \nu(dx)\right)^2+\frac{\sigma^2}{2\alpha}\right)-\frac{\gamma\sigma^2}{2\alpha \kappa}\big(1-e^{\kappa (T-t)}\big).
\end{align*}
} 
\end{Theorem}
\begin{sproof} \rm We provide here the main step of the proof and we refer to the Appendix for the details.
\textcolor{black}{\begin{itemize}
\item {\color{black}In a preliminary step, we provide an empirical investigation of the Principal problem by restricting our study to deterministic control $z$. The intuition for this restriction is that the classical model of Holmstr\"om and Milgrom leads to a deterministic $z$. More precisely, we compute explicitly the optimal continuation utility of the Principal at any time $t$ given a measure $\nu$, when he is restricted to choose a deterministic $z$, since this is a simple exercise of analysis. We denote the corresponding value by $\underline U^P(t,\nu)$, and use it as our Ansatz.}
\item In the first step, we prove that $\underline U^P$ is a smooth solution to the associated PDE \eqref{HJB:exemple} in our particular example.
\item In a second step, we use the verification result given by Theorem \ref{thm:verif:depend} to prove that $\underline U^P$ is indeed the value function of the Principal problem without any restriction, so that the Ansatz holds.
\item In the last step, we compute the optimal contract and we verify the admissibility of the revealed optimal control. 
\end{itemize}}
\end{sproof}

\subsubsection{Economic interpretations}
We focus in this section on the impact of different parameters of the model on the designed contract, the optimal effort of the Agent and the value function of the Principal. Since the optimal effort is  deterministic, $X$ and $\xi^\star$ are Gaussian, and we only need to study their expectation and variance. The salary $\xi^\star$ always decomposes into a fixed part denoted $\delta$ and a (possibly negative) variable part, indexed on the output process $X$. The one by one sensitivities to the different parameters of the model are summed up in the next tabular.

 \begin{center}
 \begin{tabular}{||c||c|c|c|c|c|c|c||}
 \hline\hline 
 & $c$& $\alpha$ & $\beta_1$ & $\beta_2$ & $\gamma$ & $\frac{\beta_1}{\alpha+\beta_1}$ & $\frac{\beta_2}{1+\beta_2}$\\ \hline\hline 
  Expectation of $\xi^*$ & $\searrow$ & $\nearrow$ & $\nearrow$  & $\nearrow$ & $=$ & $=$ & $=$ \\ \hline
  Variance of $\xi^*$ & $=$& $\nearrow$ &  $\nearrow$ & $\nearrow$ & $=$ & $=$ & $=$ \\ \hline
  Fixed salary part $\delta$ & $\searrow$ & $\searrow$ & $\searrow$ & $\searrow$ & $\nearrow$ & $\searrow$ & $=$ \\ \hline\hline
  Optimal effort of the Agent  & $\searrow$ & $\nearrow$ & $\nearrow$ & $\nearrow$ & $=$ & $=$ & $=$ \\ \hline
  Expected gain of the Principal & $\searrow$ & $\nearrow$ & $\nearrow$ & $\nearrow$ & $\searrow$ & $\nearrow$ & $=$ \\ \hline
  Variance of the terminal  gain of the Principal & $=$ & $=$ & $\nearrow$ & $\nearrow$ & $=$ & $=$ & $=$ \\ \hline
 \end{tabular}
 \end{center}

 The sensitivities of the fixed salary part and the expected gain of the Principal  with respect to a parameter, are computed when the other parameters are equal to $0$.  The last two columns of the tabular have to be understood as follows: the parameters $\alpha$ and $\beta_1$ are both playing the roles of {\it boosters} of the health of the project, and we present in the last column the impact of parameter $\beta_1$ whenever $\alpha+\beta_1$ is constant. This boils down to understanding the effect of balancing the boosting effect in the dynamics between the value of the project $X$, and the average value of the projects of the company. Similarly, the last column studies the consequences of partly replacing the effect of effort $a^\star$ by the average amount of effort $\E^\star[a^\star]$ produced in the company.
 
 \vspace{0.5em}
Let us now detail further the economic interpretation of the impact of these parameters, by focusing on the optimal effort of the Agent, the shape of the optimal contract, and the value of the game for the Principal. 
  
\paragraph{Optimal effort of the Agent}
\begin{itemize}[leftmargin=*]
\item[$(i)$] Let us first notice that the parameters $\alpha$ and $\beta_1$ play exactly the same role for the optimal effort $a^\star$, since they intervene only through their sum $\alpha+\beta_1$. Moreover, we observe that $a^\star$ is non--decreasing with respect to these two parameters. \textcolor{black}{Indeed, a higher increasing rate for each project value implies that the Agents are more eager to make efforts, which will produce more benefits in the future. This is mainly due to the fact that $\alpha$ and $\beta_1$ play the roles of the {\it boosters} of the value of each  project.} 

\item[$(ii)$] Similarly, $a^\star$ is increasing with respect to $\beta_2$. This is more surprising, since $\beta_2$ somehow measures the gain, in terms of the associated increase of the drift of $X$, that the Agent gets from the efforts of the other Agents. Hence, one could expect that the Agent may decide to work less and let the other ones do all the work for him. Therefore, at the optimum, the Principal gives sufficiently good incentives to the Agents so that at the equilibrium they select, they do not adopt any free--rider type behaviour. \textcolor{black}{{\it Ex post}, since the optimal efforts are deterministic and the Agents are identical, we observe that replacing the impact on the project dynamics of one Agent effort by the average of all efforts does not modify in any way the solution to the problem. }

\item[$(iii)$] \textcolor{black}{Notice that $a^\star$ decreases with the time $t$ (or equivalently increases with the time left until maturity $(T-t)$, whenever $\alpha+\beta_1 >0$. This was to be expected as the boosting effect (due to $\alpha$ and $\beta_1$) on the project value dynamics incentivize Agents to make more efforts at the beginning of the project and take advantage of the benefits automatically produced in the future.} 

\item[$(iv)$] As expected, $a^\star$ decreases with the cost parameter $c$: the more costly it is for the Agent to work, the more difficult it is going to be for the Principal to give him incentives to work.

\item[$(v)$] Quite surprisingly, both the volatility of the project $\sigma$, as well as the volatility penalisation, through the parameter $\gamma$, have no impact on $a^\star$. This is of course an unrealistic conclusion of the model, which is, according to us, a simple artefact of the fact that both the Principal and the Agent are risk--neutral here. As can be inferred from the general PDEs derived in Section \ref{section:PW}, this effect should generally disappear for more general situations.
\end{itemize}


\paragraph{Optimal contract $\xi^*$}
\begin{itemize}[leftmargin=*]
\item[$(i)$] The optimal contract always divides into a fixed salary $\delta$, and a random part indexed on the performances of the project $X$, corresponding to the difference between the value of the project at time $T$, and the capitalized initial value of the project with rate $\alpha+\beta_1$ (which is the rate generated by the project itself). Whenever $\beta_1=0$ and $\beta_2=0$, the Principal always transfers the totality of the project to the Agent and he does not keep any risk. This feature is classical since the Agent is risk--neutral. Indeed the same conclusions hold in Holmstr\"om and Milgrom's model \cite{holmstrom1987aggregation}.

\item[$(ii)$] Whenever $\beta_2>0$ and $\beta_1=0$, the Principal transfers to the Agents a fraction of the company, which happens to be greater than $1$, and hereby faces a random terminal gain. In order to provide proper incentives to the Agents, the Principal accepts to face a random terminal payment which surprisingly decreases with the value of the company, but that he can compensate via a reduction of the fixed salary part $\delta$. \textcolor{black}{This feature disappears in the more realistic framework where each project is impacted by a convex combination of both the effort of the representative Agent and the average efforts of all the other ones. Mathematically this corresponds to taking $\beta_2\in(0,1)$ and considering the following drift for each project:
 $$ b:(s,x,\mu,q,a)\longmapsto (1-\beta_2) a +\alpha x +\beta_2 \int_\R z dq_s(z) -\gamma V_\mu(s). $$}

\item[$(iii)$] If $\beta_1>0$, it is worth noticing that the optimal contract $\xi^\star$ is not Markovian anymore. The contract is indexed on all the values of the project $X$ over time, with a growing importance as time approaches maturity. This effect is due to the impact of the other Agents project on the dynamics of the representative one, and would not appear using the parameter $\alpha$ only.

\item[$(iv)$] The fixed salary $\delta$ is decreasing with respect to $\alpha,$ $\beta_1$ and $\beta_2$. Recall that these parameters have to be seen as boosters for the project. This reduction is compensated by the more advantageous dynamics of the project $X$, so that the average of the terminal payment $\xi^\star$ is in fact increasing with these parameters.

\item[$(v)$] On the other hand, the fixed part $\delta$ of the optimal contract, as well as the average of the terminal payment $\xi^\star$ increase with the variance penalization $\gamma$. The Principal increases the fixed part of the salary in order to compensate the negative effect on the dynamics of the project of the dispersion of the results of all the projects of the company. 


\end{itemize}

\paragraph{Value of the game for the Principal}
\begin{itemize}[leftmargin=*]
\item[$(i)$] As in the classical Principal-unique agent model, all the Agents  always obtain here their reservation utility. In our model, the interaction between the Agents does not create a situation where the Principal should provide them with a higher utility.
\item[$(ii)$] Since the Principal is risk-neutral, the value of the game for him is the expectation of his terminal gain. It is increasing with respect to the boosters $\alpha, \beta_1,\beta_2$ of the project, and is decreasing with the variance penalization $\gamma$, which is of course quite natural. 
\item[$(iii)$] From the viewpoint of the Principal, it is more interesting to have project dynamics, which are boosted through the average level of the projects, instead of the level of the project itself. The averaging effect over the level of the different projects provide a more secure dynamics for each project of the company. On the other hand, interchanging the boosting effect provided by the effort of the representative Agent, with the one provided by the average effort of all Agents does not provide any extra value for the Principal.
\end{itemize}

\subsection{Extension to risk averse Principal}

In all this section, we assume for the sake of simplicity that $\alpha> 0$. We now investigate a model in which the payoff of the Principal is penalised by the covariance between the project and the salary given to any Agent. More precisely, the Principal has to solve for some positive constants $\lambda_{X},\lambda_{\xi},\lambda_{X\xi}$
\begin{equation}\label{pb:mckeanvlasov:variance}
 U_0^P(Y_0)= \sup_{(\chi,Z)\in \mathcal X\times\mathcal Z(\chi)} \, \mathbb E^{\P^{\star}}\big[X_T-\xi\big]-\lambda_X \text{Var}_{\mathbb P^\star}(X_T)- \lambda_\xi \text{Var}_{\mathbb P^\star}(\xi)-\lambda_{X\xi} \text{Var}_{\mathbb P^\star}(X_T-\xi),
\end{equation}
where $\text{Var}_{\mathbb P^\star}$ denotes the variance under $\mathbb P^\star$. We thus have the following theorem, and we refer to the appendix for its proof,
\begin{Theorem}\label{thm:extension} Let $z^\star\in \mathcal D$ be the unique maximiser of
$$z\in \mathbb R\longmapsto h(u,z):=\bigg((1+\beta_2)\left(\frac{\left|z\right|}{c}\right)^{\frac{1}{n-1}}e^{\kappa(T-u)}-\frac{\left|z\right|^{\frac{n}{n-1}}}{c^{\frac{1}{n-1}}n}-(\lambda_\xi+\lambda_{X\xi}) \sigma^2 |z|^2+2\lambda_{X\xi} \sigma^2 z e^{\alpha(T-u)}\bigg),$$
for any $u\in [0,T]$. The optimal contract for the problem of the Principal is
\begin{align*}
\xi^\star&:= \tilde{\delta}(z^\star)-\alpha \int_0^T z^\star_t X_t dt+\int_0^T z^\star_t dX_t.
\end{align*}
for some explicit constant $\tilde{\delta}(z^\star)$, depending on $z^\star$ and the associated optimal effort of the Agent is 
$$a^\star_u:=\left(\frac{z^\star_u}{c}\right)^{\frac{1}{n-1}},\, u\in [0,T].$$
\end{Theorem}

Except in very special cases, the maximiser of the map $h$ above cannot be computed explicitly. The case $n=2$ is an exception, on which we now concentrate.
\paragraph{Particular case of quadratic cost}

\vspace{0.5em}
By considering the classical quadratic cost function $c(a):=c\frac{|a|^2}{2},\; c>0$ for the Agent, one gets
$$a^\star_t= \frac{1+\beta_2}{c\left(1+2(\lambda_\xi+\lambda_{X\xi}) c\sigma^2\right)} e^{(\alpha+\beta_1)(T-t)}+\frac{2\lambda_{X\xi}\sigma^2 }{1+2(\lambda_\xi+\lambda_{X\xi}) c\sigma^2} e^{\alpha(T-t)}$$
with optimal contract

\begin{align*}
\xi^\star&=C+\beta_1\frac{1+\beta_2}{1+2(\lambda_\xi+\lambda_{X\xi}) c\sigma^2} \int_0^T e^{\kappa(T-t)} X_t dt+\frac{1+\beta_2+2c\lambda_{X\xi} \sigma^2}{1+2(\lambda_\xi+\lambda_{X\xi}) c \sigma^2} X_T,
\end{align*}
where $C$ is an explicit constant. Besides, 
 $$
 \xi^* \; \sim \; \Nc\left( R_0 +\frac c2  \int_0^T |a^\star_t|^2dt \;;\;   \int_0^Tc^2 \sigma^2 |a^\star_t|^2dt  \right) \;.
 $$

The sensitivities to the parameters $\lambda_X,\lambda_\xi,\lambda_{X\xi}$ of the model are summed up in the next tabular.

 \begin{center}
 \begin{tabular}{||c||c|c|c|c|c|c|c||}
 \hline\hline 
 & $\lambda_X$& $\lambda_{\xi}$ & $\lambda_{X\xi}$ \\ \hline\hline 
   Optimal effort of the Agent  & $=$ & $\searrow$ & $\searrow$  \\ \hline
  Expectation of $\xi^*$ & $=$ & $\searrow$ & $\searrow$  \\ \hline
  Variance of $\xi^*$ & $=$& $\searrow$ &  $\searrow$ \\ \hline
 \end{tabular}
 \end{center}

\paragraph{Economic interpretation}

\begin{itemize}[leftmargin=*]
\item[$(i)$] The optimal effort of the agent is decreasing with the penalisation of both the variance of $\xi$ and the one of $X_T-\xi$. Indeed, lowering the value of the projects allows to reduce the variance of the output. As a consequence, the optimal contract provides incentives for Agents to provide less efforts.

\item[$(ii)$] Focus now on the particular case where $\beta_1=\beta_2=0$, corresponding to the situation where there are no interactions between the project dynamics. If the Principal criterion is only penalized by $X_T-\xi$, we retrieve a similar solution as the one obtained in the previous study. Indeed, when both Principal and Agents are risk neutral, the optimal solution without penalisation already exhibits no variance, since the Agent keeps all the risk. Hence, it is also optimal for a Principal with mean-variance criterion. 


\item[$(iii)$] Since the contract and the output process are Gaussian processes, whenever the optimal effort is deterministic, considering a mean variance criterion boils down in fact to solving an exponential utility criterion problem. 


\item[$(iv)$] The penalisation with respect to the dispersion of the project values has absolutely no effect on the optimal effort or the optimal contract. This is due to the fact that the variance of the output is not impacted by a deterministic effort and hence by the optimal one. This feature  should disappear in less linear models where the optimal effort is not deterministic.
\end{itemize}


\subsection{Links with the solution of the $N-$players' problem}\label{section:lienNjoueurs}
In this section, we study the $N$-players model by following \cite{elie2016contracting} and, through an example, we aim at studying the behaviour of the optimal contract when $N$ goes to $+\infty$ in the model introduced in Section \ref{example:meanvar} with $\beta_2=\gamma=0$. For any $x\in\mathcal C^N:=\mathcal C([0,T];\mathbb R^N)$, let $\mu^N(x)\in \mathcal P(\mathcal C^N)$ be the empirical distribution of $x$ defined by
\begin{equation}\label{def:muN}
\mu^N(x)(dz^1,\dots,dz^N):= \frac1N \sum_{i=1}^N \delta_{x^{i}}(dz^i),
\end{equation}
where for any $y\in\mathcal C$, $\delta_y(dz)\in\mathcal P(\mathcal C)$ denotes the Dirac mass at $y$. Let $A\subset \mathbb R$, we consider a drift $b^N$ in the dynamic of the $\mathbb R^N$-valued output process $X^N$, which is Markovian and defined for any $(t,x,a)\in [0,T]\times\mathbb R^N\times A^N$ by
$$b^N\big(t, x,  \mu^N(x), a\big):= a+\alpha x +\beta_1  \int_{\R^N}w\mu^N(dw),$$
which can be rewritten (see the definition of $\mu^N$) $b^N(t,x, \mu^N(x),a)= a + B^N x,$
with  $B^N:=\alpha {\rm I}_{N} + \frac{\beta_1}N \mathbf 1_{N,N}. $
We assume that the volatility of the output is given by the $N\times N-$matrix $\Sigma:= \sigma {\rm I}_N,\, \sigma>0$ and the components of the initial condition $\psi^N$ are independent and identically distributed such that $\psi^{N,i}=\psi$ in law for any $i\in \{1,\dots,N\}$. More exactly, in this model, the dynamic of the output $X^N$ is given by 
\begin{equation}\label{SDE:Nplayers}
X^N_t=\psi^N+\int_0^t b^N\big(s, X^N_s,  \mu^N(X_{s}^N),  \alpha_s\big) ds+\int_0^t\Sigma dW^{N,{\alpha}}_s,\ t\in[0,T],\ \P-a.s.,
\end{equation}
%
\noindent 
where we defined a probability measure $\mathbb P_N^{\alpha}$ which is equivalent to $\mathbb P$ by
$$ \frac{d\mathbb P_N^{\alpha}}{d\mathbb P}:= \mathcal E\left( \int_0^T \Sigma^{-1} b^N\big(t, X^N_t, \mu^N(X^N_{t}),\alpha_t\big)\cdot dW^N_t \right),$$
and an $N-$dimensional $\P^\alpha_N-$Brownian motion $W^{N,{\alpha}}$ by
$$ W^{N,{\alpha}}_t:= W^N_t -\int_0^t \Sigma^{-1} b^N\big(s, X^N_s,  \mu^N(X_{s}^N), \alpha_s\big) ds.$$
Notice that in our model, the dynamic of the $i$th project managed by the $i$th Agent is directly impacted by both the projects of other Agents and their efforts, through their empirical distribution.\vspace{0.5em}

\vspace{0.5em}
\noindent We consider that any Agent is risk neutral ($U_A(x)=x$) and he is penalised through his effort by the quadratic cost function $c$ defined for any $a\in A$ by $c(a):=ca^2/2,\, c>0$. Finally, we consider only terminal payments and we assume that there is no discount factor ($k=0$). Thus, according to Theorem 4.1 in \cite{elie2016contracting} or Theorem 3.1 in \cite{mastrolia2017moral}, we deduce that in this framework, the optimal effort for the Agents is given by the $N-$dimensional vector $a^{N,\star}(z)$ defined for any $1\leq i\leq N$ and $z\in\Mc_N(\R)$ by $(a^{N,\star}(z))^i= \frac{z^{i,i}}c.$ We now define the map $\tilde g:\mathcal M_{N}(\R)\longrightarrow \R^N$ by
\begin{align*}
\tilde g(z)&:=\left(\frac1{2c} \left|z^{1,1}\right|^2,\dots, \frac1{2c} \left|z ^{N,N}\right|^2\right)^\top.
\end{align*} 
In this case, the Hamiltonian $G$ of the HJB equation associated with the problem of the Principal (see \cite{elie2016contracting, mastrolia2017moral}) is defined by for any $(t,x,y, p_x, p_y, \gamma_x, \gamma_y, \gamma_{xy})\in [0,T]\times (\R^N)^2\times (\R^N)^2\times (\mathcal M_{N}(\R))^3$
\begin{align*}
&G(t,x,y,p_x, p_y, \gamma_x, \gamma_y, \gamma_{xy})\\
&= \sup_{z\in  \mathcal M_{N}(\R)} \Bigg\{ a^\star(z)  \cdot p_x+\tilde g(z)\cdot p_y+\frac{\sigma^2}2\text{Tr}\left[ z^\top  z \gamma_y  \right]+\sigma^2\text{Tr}\left[ z\gamma_{xy} \right] \Bigg\}+ B^N x\cdot p_x+\frac{\sigma^2}2 \text{Tr}\left[   \gamma_x\right].
\end{align*}

Thus, HJB equation associated with the Principal problem is given by
\begin{equation}\label{PDE:Nplayers}
\begin{cases}
\displaystyle 
       -\partial_t v(t,x,y)-G(t,x,y,\partial_x v,\partial_y v,\partial_{xx}v,\partial_{yy} v,\partial_{xy} v)=0,\ (t,x,y)\in [0,T)\times \R^N\times \R^N,\\[0.8em]
     \displaystyle   v(T,x,y)= U_P\left(\frac{1}{N}\left(x-y\right)\cdot {\bf 1}_N\right),\, (x,y)\in \R^N\times \R^N,
           \end{cases}
\end{equation}

Since in Section \ref{section:ex} we actually consider a risk--neutral Principal, we decided, for the sake of simplicity and tractability to assume in this example that,
\begin{Assumption}
For the $N-$players' model, the Principal is risk neutral, \textit{i.e.} $U_P(x)=x,\; x\in \R$.
\end{Assumption}

We search now for a smooth solution to HJB equation \eqref{PDE:Nplayers} of the form $v(t,x,y)= f(t,x)- \frac1N y\cdot \mathbf 1_N$, with $\partial_{x_i x_j} f=0,\, i\neq j$. Such a solution requires that $f$ satisfies the following PDE
 \begin{equation}\label{PDE:Nplayers:exemplef}
\begin{cases}
\displaystyle 
       -\partial_t f- B^N x\cdot \partial_x f-\frac{\sigma^2}2 \text{Tr}\left[ \partial_{xx} f\right]- \sup_{z\in  \mathcal M_{N}(\R)} \bigg\{ a^\star(z)  \cdot \partial_x f-\frac1N \tilde g(z)\cdot \mathbf 1_N \bigg\} =0,\ (t,x)\in [0,T)\times \R^N,\\[0.8em]
     \displaystyle   f(T,x)= \frac1N x\cdot {\bf 1}_N,\, x\in \R^N.
           \end{cases}
\end{equation}
The optimal $z$ is given by $z^\star_N :=N{\rm diag}(\partial_x f) $. 
Using a Cole--Hopf transformation 
and then Feynman--Kac's formula, we get an explicit smooth solution to PDE \eqref{PDE:Nplayers:exemplef} given by
$$ f(t,x)= e^{(\alpha+\beta_1)(T-t)} \frac1N\sum_{i=1}^Nx^i+\frac{\gamma^\star \sigma^2}{4(\alpha+\beta_1)} \left( e^{2(\alpha+\beta_1)(T-t)}-1\right).$$
Recall now that the optimal $z$ for the HJB equation \eqref{PDE:Nplayers:exemplef} is given by $z^{N,\star}_t =N {\rm diag}(\partial_x f^N)(t,x) ,\, (t,x)\in [0,T]\times \R^N$. Hence,
\begin{equation}\label{zstarN}z_t^{N,\star}=\exp((\alpha+\beta_1) (T-t)){\rm I}_N,
\end{equation}
and the value function $ V^N$ of the Principal is
\begin{equation}
\label{valeurPrincipalNplayers}V^N(t,x,y)= e^{(\alpha+\beta_1)(T-t)}\frac1N \sum_{i=1}^Nx^i +\frac{\gamma^\star \sigma^2}{4(\alpha+\beta_1)} \left( e^{2(\alpha+\beta_1)(T-t)}-1\right)- \frac1N\sum_{i=1}^N y^i.
\end{equation}

\begin{Remark}\label{Rem:icomponent}
Notice that the component of $(z_t^{N,\star})^{i,i}$ are identical. Thus, the optimal effort does not depend on the Agent. This is relevant since any Agents are supposed to be identical. 
 \end{Remark}
We thus have from Proposition 4.1 in \cite{elie2016contracting} or Theorem 3.2 in \cite{mastrolia2017moral} together with Theorem \ref{thm:optimal:ex} the following theorem, whose the proof is postponed to the appendix,
\begin{Theorem}\label{thm:NplayerMF} Assume that for any $i=1,\dots,N$, $(\lambda_0^N)^i=\lambda_0$. We have the following two properties.
\begin{itemize}[leftmargin=*]
\item[$(i)$] The optimal effort $a^{N,\star}$ of Agents in the $N-$players' model is given by
$$a^{N,\star}_t=\frac{\exp((\alpha+\beta_1) (T-t))}c \mathbf 1_N ,$$
In particular, for any $i\in\{ 1,\dots,N\}$ we have $(a^{N,\star}_t)^i= a^\star_t,$ \textit{i.e.} the optimal effort of the $i$th Agent in the $N$ players model coincides with the optimal effort of the Agent in the mean--field model.
\item[$(ii)$] The optimal contract $\xi^{N,\star}$ proposed by the Principal is
$$\xi^{N,\star}:= R_0^N-\int_0^T \frac{\exp(2\kappa (T-t))}{2c}\mathbf 1_Ndt-\int_0^T e^{\kappa (T-t)}B_N X^N_tdt+\int_0^T e^{\kappa (T-t)}dX^{N}_t,
 $$
 and for any $i\in \{1,\dots,N \}$ we have
 $$ \mathbb P^{a^{N,\star}}_N\circ\left((\xi^{N,\star})^i\right)^{-1}\underset{N\rightarrow\infty}{\overset{\rm weakly}{\longrightarrow}} \P^{a^\star}\circ(\xi^\star)^{-1}.$$
\end{itemize}

\end{Theorem}

\begin{Remark}
In our very particular example, proving the convergence of the value functions associated with the $N-$players' model and the mean field model is direct, since we can compute explicitly all the relevant quantities. In a more general case, and in view of the comprehensive investigation carried out in {\rm \cite{cardaliaguet2015master}} $($see more precisely Theorem $2.13)$, one can reasonably expect to get the convergence of the value functions of both the Principal and the Agents, and therefore the convergence of the optimal contracts $($seen as the terminal values of the continuation utilities of the Agents$)$ of the $N-$players' model to the mean field model, when $N$ goes to $+\infty$. However, the convergence of the optimal efforts in the Nash equilibria of the $N-$players' model, to the equilibria in the mean field model, remains a much harder problem, since it would necessarily involve studying the convergence of the derivatives of the value functions. We leave these very interesting questions for future research. 
\end{Remark}
 \appendix
 \section{Appendix}

\subsection{Technical proofs of Sections \ref{intro:MFM} and \ref{section:solvingMF}}

\begin{proof}[Proof of Lemma \ref{lemma:soledsrloisfixes}]
Let $(\mu,q,\alpha)\in\mathcal P(\mathcal C)\times \mathfrak P(\mathbb \R) \times\mathcal A$. We set 
$$v_t^{A}(\chi,\xi,\mu, q, \alpha):=\mathbb E^{\mathbb P^{\mu,q, \alpha}}\bigg[K_{t,T}^{X,\mu,q}U_{A}(\xi)+ \int_t^T K_{t,s}^{X,\mu,q}(u_A(s,X,\mu,q_s, \chi_s)-c(s,X,\mu,q_s,\alpha_s))ds\bigg|\mathcal F_t \bigg],$$ 
and $E_{t,T}:= \mathcal E\big( \int_t^T \sigma_s^{-1}(X)b(s,X, \mu, q_s, \alpha_s)ds\big). $
First of all, by H\"older's inequality, the definition of $\Xi$, $\mathcal A$ (recall \eqref{eq:admm}) and Assumptions $(\mathbf{C}^{p,\ell,m,\underline m})$, $(\mathbf{K})$ and $(\mathbf{U})$, we have that the process $v^A(\chi,\xi,\mu,q,\alpha)$ belongs to $\mathbb S_{\rm exp}(\mathbb R)$. Hence\footnote{Notice here that we have to go back to the probability $\P$ to apply the martingale representation theorem. Indeed, as shown in the celebrated example of Tsirelson (see Revuz and Yor \cite{revuz1999continuous}[Chapter IX, Exercise $(3.15)$]), the representation may fail under $\P^{\mu,q,\alpha}$. We would like to thank Sa\"id Hamad\`ene for pointing out this technical problem to us.}, the process $(M_t)_{t\in[0,T]}$ defined for any $t\in[0,T]$ by
$$M_t:=E_{0,t} \left(K_{0,t}^{X,\mu,q} v_t^{A}(\chi,\xi,\mu,q, \alpha)+\int_0^t K_{0,s}^{X,\mu,q}\left(u_A(s,X,\mu,q_s, \chi_s)-c(s,X,\mu,q_s, \alpha_s)\right)ds\right),$$
is a $(\mathbb P, \mathbb F)-$martingale in $\mathcal E(\mathbb R)$. We deduce that there exists an $\mathbb F-$predictable process $\widetilde Z\in\mathbb H^p(\mathbb R)$, for any $p\geq 0$, such that $M_t= M_T-\int_t^T \widetilde Z_s\sigma_s(X) dW_s.$ Applying It{\={o}}'s formula, we obtain
$$dE^{-1}_{t}=-E^{-1}_{t}b(t,X, \mu, q_t,\alpha_t)\sigma^{-1}_t(X) dW_t+  E^{-1}_t |b(t,X, \mu, q_t, \alpha_t)|^2\sigma^{-2}_t(X) dt.$$
Thus, $d(E^{-1}_t M_t)=
\widehat Z_t\sigma_t(X) dW_t^{\mu,q, \alpha,},$
with $\widehat Z_t:=  E^{-1}_{t}\left(-M_tb(t,X, \mu, q_t, \alpha_t)\sigma^{-2}_t(X)+   \widetilde Z_t\right).$ By Assumptions $(\mathbf{B}^{p,\ell,\eta})$ and $(\boldsymbol{\sigma})$, the definition of $\mathcal A$, H\"older's inequality and the fact that $\widetilde Z\in\mathbb H^p(\mathbb R)$ for any $p\geq 1$, we deduce that we also have $\widehat Z\in\mathbb H^p(\mathbb R)$ for any $p\geq 1$. 

\vspace{0.5em}
Finally, applying It\={o}'s formula again and setting 
\begin{equation*}Y_t^{\mu,q,\alpha}(\chi,\xi):=v_t^{A}(\chi, \xi,\mu,q, \alpha),\; Z_t^{\mu,q,\alpha}(\chi,\xi):= \exp\left(\int_0^t k(s,X,\mu,q_s) ds\right) \widehat Z_t,\end{equation*}
we deduce that $(Y_t^{\mu,q,\alpha}(\chi,\xi), Z_t^{\mu,q,\alpha}(\chi,\xi))$ is a solution to BSDE \eqref{edsr:agent:loisfixes}. Moreover, recalling that since $\mathcal F_0$ is not trivial, the Blumenthal $0-1$ law does not hold here, we cannot claim that $Y_0^{\mu,q,\alpha}(\chi,\xi)$ is a constant, and it {\it a priori} depends on the canonical process $\psi$. We therefore have 
$$\mathbb E[Y_0^{\mu,q,\alpha}(\chi,\xi)]=\int_\R Y_0^{\mu,q,\alpha}(\chi,\xi)(x)\lambda_0(dx)=v_0^{A}(\chi, \xi,\mu,q, \alpha).$$

\textcolor{black}{The uniqueness is a classical under Assumption $(\mathbf{K})$ and $(\mathbf{B}^{p,\ell,\eta})$, since the generator is uniformly Lipschitz in $y$, and linear in $z$, with a linearity process which is, by definition of $\mathcal A$, sufficiently integrable, and we refer to \cite{el1997backward} for instance.}
\end{proof}

\begin{proof}[Proof of Lemma \ref{lemma:maxg}]
Under Assumptions $(\mathbf{B}^{p,\ell,\eta})$, $\mathbf{(K)}$ and $(\mathbf{C}^{p,\ell,m,\underline m})$, for any fixed $(s,x,\mu,q,y,z,\chi)\in [0,T]\times \mathcal C\times \mathcal P(\mathcal C)\times \mathfrak P(\mathbb R)\times \mathbb R\times\mathbb R\times\mathbb R_+$, we have that the map $a\longmapsto g(s,x, y, z, \mu,q,a,\chi)$ is continuous. Therefore, if $A$ is bounded, it is a compact subset of $\mathbb R$ and the existence of a bounded maximiser is obvious.

\vspace{0.5em}
Assume now that $A$ is unbounded. Then, it is also clear that for any fixed $(s,x,\mu,q,y,z,\chi)\in [0,T]\times \mathcal C\times \mathcal P(\mathcal C)\times \mathfrak P(\mathbb R)\times \mathbb R\times\mathbb R\times\mathbb R_+$ the map $a\longmapsto g(s,x, y, z, \mu,q,a,\chi)$ is coercive, in the sense that it goes to $-\infty$ as $|a|$ goes to $+\infty$. Therefore, the existence of a maximiser is obvious. Besides, since neither $k$ nor $u_A$ depend on $a$, then any maximiser clearly does not depend on $\chi$ and $y$. Let now $a^\star(s,x,z,\mu,q)\in{\rm{argmax }}_{a\in A}\,g(s,x, y, z, \mu,q,a,\chi)$. If it happens to belong to the boundary of $A$, then it is automatically bounded. Otherwise, if this is an interior maximiser, it satisfies the first order conditions
\begin{align}\label{firstordercondition}
z\partial_ab(s, x, \mu, q_s,a^\star(s,x,z,\mu,q))=\partial_ac(s,x,\mu,q_s,a^\star(s,x,z,\mu,q)).
\end{align} 
Thus, by our assumptions on the derivatives of $c$ and $b$, there exists some positive constant $\kappa$ such that
\begin{align*}
|a^\star(s,x,z,\mu,q)|^{\underline m}
&\leq \kappa|z| |\partial_a b(s, x, \mu, q_s,a^\star(s,x,z,\mu,q))|\\
&\leq \kappa|z|\bigg( 1+b^1(\No{x}_{s,\infty}) +\left(\int_{\mathcal C} \|w\|_{s,\infty}^p \mu(dw)\right)^{\frac1p}+\left(\int_{\mathbb R} |w|^p q(dw)\right)^{\frac1p}\bigg)\\
&\hspace{0.9em}+\kappa|z||a^\star(s,x,z,\mu,q)|^{\ell-1}.
\end{align*}
Therefore, if $a^\star$ is unbounded, there exists $D>0$ such that
\begin{align*}
&|a^\star(s,x,z,\mu,q)|\\
&\leq D |z|^{\frac{1}{\underline m+1-\ell}}\bigg( 1+|b^1(\No{x}_{s,\infty})|^{\frac{1}{\underline m+1-\ell}} +\bigg(\int_{\mathcal C} \|w\|_{s,\infty}^p \mu(dw)\bigg)^{\frac{1}{p(\underline m+1-\ell)}}+\bigg(\int_{\mathbb R} |w|^p q(dw)\bigg)^{\frac{1}{p(\underline m+1-\ell)}}\bigg).
\end{align*}
Hence the desired result for $a^\star$. The growth for $g$ at $a^\star(s,x,z,\mu,q)$ is immediate from our assumptions.
\end{proof}

\begin{proof}[Proof of Theorem \ref{thm:agent}]
The proof of this theorem is similar to the proof of Theorem 4.1 in \cite{elie2016contracting} and based on the so--called martingale optimality principle. We prove it in two steps.\vspace{0.5em}

\noindent \textbf{Step 1: a solution to $\mathbf{(MFG)}(\chi,\xi)$ provides a solution to BSDE \eqref{edsr:agent:max}}.\vspace{0.3em}

\noindent Let $(\mu,q,\alpha^{\star})$ be a solution to $\mathbf{(MFG)}(\chi,\xi)$. Let $\tau\in \mathcal T_{[0,T]}$, the set of $\mathbb F-$stopping times valued in $[0,T]$. We define the following family of random variables $V^A_\tau(\chi,\xi,\mu,q):=\underset{\alpha\in\mathcal A}{\esup}\, v_\tau^{A}(\chi,\xi,\mu, q, \alpha), $
where 
$$v_\tau^{A}(\chi,\xi,\mu, q, \alpha)=\mathbb E^{\mathbb P^{\mu,q, \alpha}}\left[\left.K_{\tau,T}^{X,\mu,q}U_{A}(\xi)+ \int_\tau^T K_{\tau,s}^{X,\mu,q}(u_A(s,X,\mu,q_s, \chi_s)-c(s,X,\mu,q_s,\alpha_s))ds\right|\mathcal F_\tau \right].$$
From for instance Theorem 2.4 in \cite{karoui2013capacities2} as well as the discussion in their Section 2.4.2, we recall that this family satisfies the following dynamic programming principle
\begin{align*}&V^A_\tau(\chi,\xi,\mu,q)\\
&= \underset{\alpha\in\mathcal A}{\esup}\, \mathbb E^{\P^{\mu,q,\alpha}}\left[\left.K_{\tau,\theta}^{X,\mu,q}V_\theta(\chi,\xi,\mu,q)+\int_\tau^\theta K_{\tau,s}^{X,\mu,q}(u_A(s,X,\mu,q_s, \chi_s)-c(s,X,\mu,q_s,\alpha_s))ds\right|\mathcal F_\tau\right], 
\end{align*} for any $\theta\in \mathcal T_{[0,T]}$ with $\tau\leq \theta, \; \mathbb P-a.s.$ We thus notice that for any $\alpha\in \mathcal A$, the family 
$$\left(\mathcal K_{0,\tau}^{X,\mu,q}V^A_\tau(\chi,\xi,\mu,q) +\int_0^\tau K_{0,s}^{X,\mu,q}(u_A(s,X,\mu,q_s, \chi_s)-c(s,X,\mu,q_s,\alpha_s))ds\right)_{\tau\in \mathcal T_{[0,T]}},$$
is a $\mathbb P^{\mu,q,\alpha}-$super--martingale system which can be aggregated (see \cite{della1981sur}) by a unique $\mathbb F-$optional process coinciding with
$$(\mathcal M^\alpha_t)_{t\in [0,T]}:= \left(\mathcal K_{0,t}^{X,\mu,q}V^A_t(\chi,\xi,\mu, q) +\int_0^t K_{0,s}^{X,\mu,q}\left(u_A(s,X,\mu,q_s, \chi_s)-c(s,X,\mu,q_s,\alpha_s)\right)ds\right)_{t\in [0,T]},$$
which is therefore a $\mathbb P^{\mu,q,\alpha}-$super--martingale for any $\alpha\in \mathcal A$. 

\vspace{0.5em}
We now check that $\mathcal M^{\alpha^{\star}}$ is a $\mathbb P^{\mu,q,\alpha^{\star}}-$uniformly integrable martingale. Since $(\mu,q,\alpha^{\star})$ is a solution to $\mathbf{(MFG)}(\chi,\xi)$, we have
\begin{align*}
V_0^A(\chi,\xi,\mu,q)&=v_0^A (\chi,\xi,\mu,q,\alpha^{\star})\\
&= \mathbb E^{\mathbb P^{\mu,q, \alpha^{\star}}}\left[K_{0,T}^{X,\mu,q}U_{A}(\xi)+ \int_0^T K_{0,s}^{X,\mu,q}(u_A(s,X,\mu,q_s, \chi_s)-c(s,X,\mu,q_s,\alpha^{\star}_s))ds \right].
\end{align*}
By the super--martingale property, we thus have
\begin{align*}
V_0^A(\chi,\xi,\mu,q)\geq\mathbb E^{\mathbb P^{\mu,q, \alpha^{\star}}}\left[\mathcal M_t^{\alpha^{\star}} \right]\geq \mathbb E^{\mathbb P^{\mu,q, \alpha^{\star}}}\left[\mathcal M_T^{\alpha^{\star}} \right]= V_0^A(\chi,\xi,\mu,q).
\end{align*}
Hence, $\mathcal M_t^{\alpha^{\star}}= \mathbb E^{\mathbb P^{\mu,q, \alpha^{\star}}}\left[\left.\mathcal M_T^{\alpha^{\star}}\right| \mathcal F_t\right],$ which shows that $\mathcal M^{\alpha^{\star}}$ is a $\mathbb P^{\mu,q,\alpha^{\star}}-$martingale, which is uniformly integrable under Assumptions $(\mathbf{B}^{p,\ell,\eta})$, $(\boldsymbol{\sigma})$, $\mathbf{(K)}$ and $(\mathbf{C}^{p,\ell, m,\underline m})$. Furthermore, it is also clear thanks to H\"older's inequality that $V^A$ belongs to $\mathcal E(\mathbb R)$, since $\alpha^{\star}\in\mathcal A$. By Lemma \ref{lemma:soledsrloisfixes}, we deduce that there exists some process $Z^{\mu,q,\alpha^{\star}}\in \mathbb H_{\rm exp}^{\lambda r}(\mathbb R)$ (recall that the solution of the mean--field game is assumed to be in ${\rm MFI}^r(\chi,\xi)$) such that
\begin{align*}
\nonumber V^A_t( \chi,\xi,\mu,q)=&\ U_A(\xi)+\int_t^T g\big(s,X, V^A_s( \chi,\xi,\mu,q), Z_s^{\mu,q,\alpha^{\star}}(\chi,\xi), \mu,q_s,\alpha^{\star}_s,\chi_s\big)  ds\\
&-\int_t^T Z_s^{ \mu,q, \alpha^{\star}}(\chi,\xi) \sigma_s(X) dW_s.
\end{align*}
We thus deduce that for any $\alpha\in \mathcal A$, we have
\begin{align*}
\nonumber \mathcal M_t^\alpha=&\ \mathcal K_{0,T}^{X,\mu,q}U_A(\xi)+\int_t^T g\big(s,X, V^A_s( \chi,\xi,\mu,q), Z_s^{\mu,q,\alpha^{\star}}(\chi,\xi), \mu,q_s,\alpha^{\star}_s,\chi_s\big)  ds\\
&-\int_t^T Z_s^{ \mu,q, \alpha^{\star}}(\chi,\xi) \sigma_s(X) dW_s^{\mu,q,\alpha}.
\end{align*}
Since $\mathcal M_t^\alpha$ has to be a $\mathbb P^{\mu,q,\alpha}-$super--martingale, we have necessarily for any $\alpha\in \mathcal A$ 
$$ g\big(s,X, V^A_s( \chi,\xi,\mu,q), Z_s^{\mu,q,\alpha^{\star}}(\chi,\xi), \mu,q_s,\alpha^{\star}_s,\chi_s\big) \geq g\big(s,X, V^A_s( \chi,\xi,\mu,q), Z_s^{\mu,q,\alpha^{\star}}(\chi,\xi), \mu,q_s,\alpha_s,\chi_s\big).$$
By setting $Y^\star(\chi,\xi):= V(\chi,\xi,\mu,q),$ $Z^\star (\chi,\xi):= Z^{\mu,q,\alpha^{\star}}(\chi,\xi), $
we have finally proved that the quadruple $(Y^\star(\chi,\xi),Z^\star (\chi,\xi),\mu,q)$ is a solution to BSDE \eqref{edsr:agent:max} and $\alpha^{\star}\in \mathfrak A^{X,Z^\star,\mu,q}$.

\vspace{0.8em}
\noindent \textbf{Step 2: a solution to BSDE \eqref{edsr:agent:max} provides a solution to $\mathbf{(MFG)}(\chi,\xi)$}.\vspace{0.3em} 

\noindent Let now $(Y^\star(\chi,\xi),Z^\star(\chi,\xi),\mu,q)$ be a solution to BSDE  \eqref{edsr:agent:max}. Recall from Lemma \ref{lemma:maxg} together with a measurable selection argument that there exists a process $a^\star(\cdot,X,Z^\star(\chi,\xi),\mu,q_\cdot)\in \mathbb R$ such that
$$a_s^\star:=a^\star(s,X,Z^\star_s(\chi,\xi),\mu,q_s)\in\underset{a\in \mathcal A}{\rm{argmax }}\,g(s,X, Y^\star_s(\chi,\xi), Z^\star_s(\chi,\xi), \mu,q_s,a).$$
Following the computations of Step 1, we obtain that for any $\alpha\in \mathcal A$,
$$(\widetilde{\mathcal M}_t^\alpha)_{t\in [0,T]}:=  \left(\mathcal K_{0,t}^{X,\mu,q}Y_t^\star(\chi,\xi) +\int_0^t K_{0,s}^{X,\mu,q}(u_A(s,X,\mu,q_s, \chi_s)-c(s,X,\mu,q_s,\alpha_s))ds\right)_{t\in [0,T]},$$
is a $\mathbb P^{\mu,q,\alpha}-$super--martingale and a $\mathbb P^{\mu,q,a^\star}-$martingale, which is uniformly integrable since $Y^\star(\chi,\xi)\in \mathcal E(\mathbb R) $. Remember that by Lemma \ref{lemma:soledsrloisfixes} we have
$$v_0^A(\chi,\xi,\mu,q,a^\star)=\mathbb E^{\P}[Y_0^\star(\chi,\xi)]=\mathbb E^{\P^{\mu,q,a^\star}}[Y_0^\star(\chi,\xi)],$$
since the only randomness in $Y_0^\star(\chi,\xi)$ comes from the canonical process $\psi$, so that it does not matter which probability measure you use in the expectation. We thus have for any $\alpha\in \mathcal A$, using Fubini's Theorem and the fact that $\widetilde M_0^\alpha$ is actually independent of $\alpha$ by construction
\begin{align*}
v_0^A(\chi,\xi,\mu,q,a^\star)=\mathbb E^{\P^{\mu,q,a^\star}}[Y_0^\star(\chi,\xi)]=\mathbb E^{\mathbb P^{\mu,q,a^\star}}\Big[ \widetilde M_T^{a^\star}\Big]= \mathbb E^{\mathbb P^{\mu,q,a^\star}}\Big[\widetilde M_0^{\alpha}\Big]= \mathbb E^{\mathbb P^{\mu,q,\alpha}}\Big[\widetilde M_0^\alpha\Big]&\geq   \mathbb E^{\mathbb P^{\mu,q,\alpha}}\Big[ \widetilde M_T^\alpha\Big]\\
&= v_0^A(\chi,\xi,\mu,q,\alpha).
\end{align*}
It therefore means that we have $V_0^A(\chi,\xi,\mu,q)=v_0^A(\chi,\xi,\mu,q,a^\star)$. Furthermore, we have by definition that $\mathbb P^{\mu,q, a^\star}\circ (X)^{-1}=\mu $ and $\mathbb P^{\mu,q,a^\star}\circ (a^\star_t)^{-1}=q_t $ for Lebesgue almost every $t\in[0,T]$. Hence, it simply remains to prove that $a^\star$ is indeed in $\mathcal A$ to deduce that $(\mu,q,a^\star)$ is a solution to $\mathbf{(MFG)}(\chi,\xi)$. Recall that $Z^\star\in \mathbb H_{\rm exp}^{\lambda r }(\mathbb R)$. We then obtain for any $h\geq 1$ and for some constants $\kappa,\tilde\kappa>0$ and by denoting $\underline r$  the conjugate of $r$
 \begin{align*}
\mathbb E\left[\exp\left(h\int_0^T|a^\star_t|^{\ell+m}dt\right)\right]&\leq \kappa\mathbb E\left[\exp\left(h\tilde\kappa\int_0^T |Z_t^\star|^{\lambda }\left(1+b^1(\| X\|_{t,+\infty})^\lambda \right)dt\right)\right]\\
&\leq\kappa\mathbb E\left[\exp\left(h\tilde\kappa\left(\int_0^T |Z_t^\star|^{\lambda r} dt \right)^\frac1r \left(\int_0^T|b^1(\| X\|_{t,+\infty})|^{\lambda \underline r} \right)dt\right)^\frac{1}{\underline r}\right] \\
&\leq\kappa\mathbb E\left[\exp\left(h\tilde\kappa \int_0^T |Z_t^\star|^{\lambda r} dt \right)\right]<+\infty,
\end{align*}
by using Lemma \ref{lemma:maxg}, H\"older's Inequality, Young's Inequality and the definition of $Z^\star$. We now show that $a^\star\in \mathcal A_{\tilde{\varepsilon}}$ for some $\tilde\varepsilon>0$. Recall that $\ell\leq m$. Then, there exists a positive constant $\kappa>0$ which may vary from line to line such that for any $h,k>1$ we have by H\"older's inequality
\begin{align*}
&\mathbb E\left[\exp\left(\frac{k}{2}\int_0^T|\sigma_t^{-1}(X) b(t, X, \mu, q_t,a^\star_t)|^2dt\right)\right]\\
&\leq \kappa\mathbb E\left[\exp\left(\frac{3M^2k}{2} \int_0^T \left(C^2|a^\star_t|^{2\ell}+|b^0(\|X\|_{t,\infty})|^2dt\right) \right)\right]\\
&\leq \kappa  \mathbb E\left[\exp\left(\frac{3M^2C^2kh}{2(h-1)}\int_0^T|a^\star_t|^{2\ell} dt\right)\right]^{1-\frac1h} \mathbb E\left[\exp\left(\frac{3hk M^2}{2}\int_0^T|b^0(\| X\|_{t,\infty})|^2 dt\right)\right]^\frac1h,
\end{align*}
which is finite as soon as $ph\leq\eta$.

\vspace{0.5em}
Thus, since $\eta>1$, we can choose $k>1$, so that by L\'epingle and M\'emin \cite{lepingle1978sur}[Th\'eor\`eme $a)$], we deduce that there exists $\tilde\varepsilon>0$ such that
$$\mathbb E\left[\left(\mathcal E\left( \int_0^T   \sigma_t^{-1}(X) b(t, X, \mu, q_t,a^\star_t)dW_t \right)\right)^{1+\tilde\eps}\right]<+\infty.$$
\end{proof}

\begin{proof}[Proof of Corollary \ref{corollary:xi}]
Let $(\chi,\xi)\in \Xi$. From Theorem \ref{thm:agent}, there exists $(Y^\star,Z^\star,\mu,q)\in \mathcal E(\mathbb S)\times \mathbb H^{\lambda r}_{\rm exp}(\mathbb R)\times\mathcal P(\mathcal C)\times\mathfrak P(\mathbb R)$ such that $(Y^\star,Z^\star,\mu,q)$ is a solution to BSDE \eqref{edsr:agent:max}. Besides, the proof of the Step 1 of Theorem \ref{thm:agent} shows that $\mathbb E[Y_0^\star]=V_0^A(\chi,\xi,\mu,q)\geq R_0$. Thus, we deduce that the system \eqref{systemeEDS} admits a solution with parameters $Y_0:=\mathbb E[Y_0^\star]$ and $Z:=Z^\star$. Thus $(\chi,\xi)\in \widehat \Xi$.

\vspace{0.5em}
\noindent Conversely, let $(\chi,\xi)\in \widehat \Xi$ with $\xi=U_A^{(-1)}\big(Y_T^{Y_0,Z}(\chi)\big)$ where the quadruple $(Y^{Y_0,Z}(\chi), Z, \mu,q)$ is a solution to the system \eqref{systemeEDS} with $Y_0\geq R_0$ and $Z\in\mathcal Z(\chi)$. Then, according to Theorem \ref{thm:agent}, we deduce that $(\chi,\xi)\in \Xi$.
\end{proof}

 \subsection{Technical proofs of Section \ref{section:ex}}
 
 \begin{proof}[Proof of Theorem \ref{thm:optimal:ex}]
 
 For the considered model, we will show that there exists a smooth solution to HJB equation \eqref{HJB:exemple} and that the optimal $z$ is in $\mathcal D$. This leads us to set the following {\it ansatz}

\begin{Ansatz}\label{ansatz}
The optimal $z$ for the HJB equation \eqref{HJB:exemple} is in $\mathcal D$.
\end{Ansatz}
Considering this {\it ansatz}, and by denoting $\underline U(t,\nu)$ the correspond dynamic version of the problem of the Principal with $Z$ restricted to $\mathcal D$, we obtain for any $(t,\nu)\in [0,T]\times \mathcal P(\mathcal C^2)$  
\begin{align}
\nonumber\underline{U}^P(t,\nu)=&\ \sup_{z\in \mathcal D} \left\{\int_t^TH(u,z_u) du\right\}+\int_\R x\nu^1(dx) e^{\kappa (T-t)}-\int_\R x\nu^2(dx)\\
\label{valeurgamma}&-\gamma\int_t^T  e^{\kappa (T-u)} \left(e^{2\alpha (u-t)}V_{\nu^1}(t)+\frac{\sigma^2}{2\alpha}(e^{2\alpha(u-t)}-1)\mathbf{1}_{\alpha>0}+ \sigma^2(u-t) \mathbf 1_{\alpha=0}\right) du,
\end{align}
where
$$H(u,z):= (1+\beta_2)\left(\frac{\left|z\right|}{c}\right)^{\frac{1}{n-1}}e^{\kappa(T-u)}-\frac{\left|z\right|^{\frac{n}{n-1}}}{c^{\frac{1}{n-1}}n}. $$
Thus,  $z^\star_u:=(1+\beta_2)e^{\kappa(T-u)}$ is optimal for this sub-optimal problem. As for the corresponding value, we need to distinguish several cases.

\vspace{0.5em}
$\bullet$ If $\alpha=\beta_1=0$, we get
\begin{align*}\underline{U}^P(t,\nu)=&\ (T-t)\left( \frac{(1+\beta_2)^{\frac n{n-1}}}{c^\frac1{n-1}}\frac{n-1}n\right)+\int_\R x\nu^1(dx) -\int_\R x\nu^2(dx)\\
&-\gamma(T-t)\left(\int_\R |x|^2 \nu^1(dx)-\left(\int_\R x \nu^1(dx)\right)^2\right)-\gamma\sigma^2\frac{(T-t)^2}2.
\end{align*}
$\bullet$ If $\alpha=0$ and $\beta_1>0$, we get
\begin{align*}\underline{U}^P(t,\nu)=&\ \int_t^T(1+\beta_2)^{\frac{n}{n-1}}\frac{e^{\beta_1\frac{n}{n-1}(T-u)}}{c^\frac{1}{n-1}} \left(1-\frac1n \right)du+ \int_\mathbb R x \nu^1(dx) e^{\beta_1(T-t)}-\int_\mathbb R x \nu^2(dx)\\
&-\gamma\frac{e^{\beta_1(T-t)}-1}{\beta_1}\left(\int_\R |x|^2 \nu^1(dx)-\left(\int_\R x \nu^1(dx)\right)^2\right)-\gamma\sigma^2\int_t^T e^{\beta_1(T-u)}(u-t) du.
\end{align*}
$\bullet$ If $\alpha>0$ and $\alpha \neq \beta_1$, we get
\begin{align}
\nonumber\underline{U}^P(t,\nu)=&\ \int_t^T(1+\beta_2)^{\frac{n}{n-1}}\frac{e^{\kappa\frac{n}{n-1}(T-u)}}{c^\frac{1}{n-1}} \left(1-\frac1n \right)du+ \int_\mathbb R x \nu^1(dx) e^{\kappa(T-t)}-\int_\mathbb R x \nu^2(dx)\\
&-\frac{\gamma}{2\alpha-\kappa}(e^{2\alpha (T-t)}-e^{\kappa (T-t)}) \left(\int_\R |x|^2 \nu(dx)-\left(\int_\R x \nu(dx)\right)^2+\frac{\sigma^2}{2\alpha}\right)-\frac{\gamma\sigma^2}{2\alpha \kappa}(1-e^{\kappa (T-t)}).
\label{underlineU}
\end{align}
$\bullet$ If $\alpha>0$ and $\alpha=\beta_1$ we have
\begin{align*}
\underline{U}^P(t,\nu)=&\ \int_t^T(1+\beta_2)^{\frac{n}{n-1}} \frac{e^{\kappa\frac{n}{n-1}(T-u)}}{c^\frac{1}{n-1}} \left(1-\frac1n \right)du+ \int_\mathbb R x \nu^1(dx) e^{\kappa(T-t)}-\int_\mathbb R x \nu^2(dx)\\
&-\gamma e^{2\alpha (T-t)} (T-t)\left(\int_\R |x|^2 \nu^1(dx)-\left(\int_\R x \nu^1(dx)\right)^2+\frac{\sigma^2}{2\alpha}\right)-\frac{\gamma\sigma^2}{4\alpha^2 \kappa}(1-e^{2\alpha(T-t)}).
\end{align*}

We only prove the result when $\alpha\neq \beta_1$ and $\alpha,\beta_1>0$ since the other cases can be obtained similarly. The proof is divided in three steps. In the first one, we show that $ \underline{U}^P(t,\nu)$ is a smooth solution to the HJB equation \eqref{HJB:exemple}. In the second step, we provide a verification theorem by adapting the proof of \cite[Theorem 4.1]{pham2015bellman}, and in the last step we compute explicitly the optimal contract $\xi^\star$.

\paragraph{Step 1: a smooth solution to the HJB equation \eqref{HJB:exemple}.} 
Direct computations show that
%
for any $z\in \tilde {\mathcal Z}_p$
\begin{align*}
&\int_\mathbb R \mathbb L_t^z \underline{U}^P(t,\nu)(x) \nu(dx)\\
& = \int_\R  \left(\left( \frac{\left|z(t,x,\nu)\right|}{c}\right)^{\frac{1}{n-1}}(1+\beta_2) e^{\kappa(T-t)}- \left(\frac{\left|z(t,x,\nu)\right|^{n}}{cn^{n-1}}\right)^{\frac{1}{n-1}}\right)\nu(dx)+ \kappa\int_\R x\nu^1(dx)e^{\kappa(T-t)} \\
&\hspace{0.9em}-\left(2\frac{\gamma\alpha}{2\alpha-\kappa}(e^{2\alpha (T-t)}-e^{\kappa (T-t)})+\gamma e^{\kappa(T-t)} \right)\left( \int_\mathbb R |x|^2 \nu^1(dx)-\left(\int_\mathbb R x\nu^1(dx)\right)^2\right)\\
&\hspace{0.9em}-\frac{\gamma\sigma^2}{2\alpha-\kappa}(e^{2\alpha (T-t)}-e^{\kappa (T-t)}).
\end{align*}
Notice that for any $(t,x,\nu)\in [0,T]\times\R^2\times  \mathcal P(\mathcal C^2)$, the maximum of
$$z\in \tilde{\mathcal Z_p}\longmapsto \left( \frac{\left|z(t,x,\nu)\right|}{c}\right)^{\frac{1}{n-1}}(1+\beta_2) e^{\kappa(T-t)}- \left(\frac{\left|z(t,x,\nu)\right|^{n}}{cn^{n-1}}\right)^{\frac{1}{n-1}}$$ is clearly $z^\star_t:=(1+\beta_2)e^{\kappa(T-t)}$. Thus, we have 
\begin{align*}
&\sup_{z\in \tilde {\mathcal Z}_p}\left\{ \int_{\R^2}  \left(\left( \frac{\left|z(t,x,\nu)\right|}{c}\right)^{\frac{1}{n-1}}(1+\beta_2) e^{\kappa(T-t)}- \left(\frac{\left|z(t,x,\nu)\right|^{n}}{cn^{n-1}}\right)^{\frac{1}{n-1}}\right)\nu(dx)\right\}\\
&=(1+\beta_2)^{\frac{n}{n-1}} \frac{e^{\kappa\frac{n}{n-1}(T-t)}}{c^\frac{1}{n-1}} \left(1-\frac1n \right).
\end{align*}

Thus, the HJB Equation \eqref{HJB:exemple} is satisfied by $\underline U^P$.

\paragraph{Step 2: Verification result.} 
Using the fact that $\underline U^P$ solves \eqref{HJB:exemple}, we thus get from the verification Theorem \ref{thm:verif:depend}
that $z^\star_t:=(1+\beta_2)e^{\kappa(T-t)}$ is a maximiser for the HJB equation \eqref{HJB:exemple} in $\mathcal D$, then in $\tilde {\mathcal Z}^p$. Thus, we deduce that $\underline U^P_0=U^P_0$, \textcolor{black}{so that the Ansatz holds,} with the optimal control $z^\star_t:=(1+\beta_2)e^{\kappa(T-t)}, \, t\in [0,T]$.

\paragraph{Step 3: computation of the optimal contract and admissibility.} Notice that the optimal effort of the Agent is $a^\star_u:= a^\star_u(z^\star)=(1+\beta_2)^{\frac{1}{n-1}}\left( \frac{e^{(\alpha+\beta_1) (T-u)}}{c}\right)^{\frac{1}{n-1}}$. By remembering the discussion of Section \ref{section:admissibility}, their exists a solution $(X,Y^{R_0,z^\star},\mu^\star,q^\star)$ to the system \ref{systemeEDS} and according to Corollary \ref{corollary:xi},
 $(\mu^\star,q^\star,a^\star)$ is a solution to \textbf{(MFG)}$(\xi^\star)$ in $\text{MF}^r(\xi^\star)$ for any $r>1$ with $\xi^\star=Y_T^{R_0,z^\star}$. 
 
\vspace{1em}
Denote now $\mathbb P^\star:=\mathbb P^{\mu^\star,q^\star,a^\star}$ and $W^\star:=W^{\mu^\star,q^\star,a^\star}$. Let $t\in [0,T]$, we set 
$$f(t):= \mathbb E^{\star}[X_t],\ g(t):=\mathbb E^{\star}\big[|X_t|^2\big],\ z_n(t):=\left(\frac{\left|z^\star_t\right|}{c}\right)^{\frac{1}{n-1}},\  \text{for any $t\in [0,T], \, Z\in \mathcal C$}.$$ 
Thus, applying It\=o's formula, we obtain the following system of ODEs for any $0\leq t\leq s\leq T$
\begin{empheq}[left=\empheqlbrace]{align}
    \label{ode:f:bis}  f'(t)&= (1+\beta_2)z_n(t)+(\alpha+\beta_1) f(t)- \gamma g(t)+\gamma |f(t)|^2,\ f(0)=\int_{\R}x\lambda_0(dx),\\
     \label{ode:g:bis}    g'(t)&=2\alpha g(t)+ 2f(t)(f'(t)-\alpha f(t))+\sigma^2,\ g(0)=\int_{\R}x^2\lambda_0(dx).
\end{empheq}

Using \eqref{ode:g:bis} we have
$$g(t)= \big(g(0)-f^2(0)\big)e^{2\alpha t}+f^2(t)+\frac{\sigma^2}{2\alpha}(e^{2\alpha t }-1) . $$

The ODE \eqref{ode:f:bis} thus become for any $0\leq t\leq T$
\begin{align}
\label{ode:f:linear}f'(t)&=\kappa f(t)+H_n(t),\ f(0)=\int_{\R}x\lambda_0(dx),
\end{align}
with $H_n(t):= (1+\beta_2)z_n(t)- \gamma \frac{\sigma^2}{2\alpha}(e^{2\alpha t}-1)-\gamma \big(g(0)-f^2(0)\big)e^{2\alpha t}.$
Thus, the linear ODE \eqref{ode:f:linear} has a unique solution given by
\begin{align*}
f(t)&=\int_{\R}x\lambda_0(dx) e^{\kappa t }+\int_0^t e^{\kappa(t-u)} H_n(u) du
=\lambda_1 e^{\kappa t}+ \lambda_2 e^{-\frac{\kappa }{n-1}t}+\lambda_3 e^{2\alpha t }+\lambda_4,
\end{align*}
with
\begin{align*}
\lambda_1&:=\int_{\R}x\lambda_0(dx)+(1+\beta_2)\left(\frac{(1+\beta_2)e^{\kappa T}}{c}\right)^{\frac{1}{n-1}}\frac{n-1}{\kappa n}+\frac{\gamma\sigma^2}{\kappa(2\alpha-\kappa)}+\frac{\gamma}{2\alpha-\kappa}\big(g(0)-f^2(0)\big),\\
\lambda_2&:= -(1+\beta_2)\left(\frac{(1+\beta_2)e^{\kappa T}}{c}\right)^{\frac{1}{n-1}}\frac{n-1}{\kappa n},\
 \lambda_3:=-\frac{\gamma\sigma^2}{2\alpha(2\alpha-\kappa)}-\frac{\gamma\big(g(0)-f^2(0)\big)}{2\alpha-\kappa},\
  \lambda_4:=-\frac{\gamma\sigma^2}{2\alpha\kappa}.
\end{align*}
Thus, the optimal contract is given, after tedious but simple computations
\begin{align*}
\xi^\star&=
\delta-\alpha (1+\beta_2) \int_0^T e^{\kappa(T-t)} X_t dt +(1+\beta_2) \int_0^T e^{\kappa(T-t)} dX_t,
\end{align*}

with 
\begin{align}\label{delta_optimal}
\nonumber \hspace{-1em}\delta
:=&\ R_0-\frac{\left(1+\beta_2-\frac1n\right)(1+\beta_2)^{\frac{n}{n-1}}}{n(\alpha+\beta_1)c^{\frac{1}{n-1}}}\left( e^{(\alpha+\beta_1)\frac{n}{n-1}T} -1\right)\\
\nonumber&-(1+\beta_2)e^{(\alpha+\beta_1)T}\Big[ T\beta_1 \left(f(0) + \frac{\gamma\sigma^2}{(\alpha+\beta_1)(\alpha-\beta_1) }+\frac{\gamma V_{\lambda_0}}{\alpha-\beta_1}\right)+\frac12 \frac{\gamma\sigma^2\left(1-e^{-(\alpha+\beta_1)T} \right)}{(\alpha+\beta_1)^2}\\
\nonumber & -\frac{\gamma \alpha \left( \frac{\sigma^2}{2\alpha} +V_{\lambda_0}\right)}{(\alpha-\beta_1)^2} \left(e^{(\alpha-\beta_1)T}-1 \right) +\frac{(1+\beta_2)(n-1)\beta_1}{nc^{\frac{1}{n-1}}(\alpha+\beta_1)} e^{\frac{\alpha+\beta_1}{n-1} T} \left(T-\frac{n-1}{n(\alpha+\beta_1)} \left(1-e^{-\frac{n}{n-1}(\alpha+\beta_1)T } \right)\right)\Big].\\
\end{align}

 \end{proof}
 
 \begin{proof}[Proof of Theorem \ref{thm:extension}]
First notice that since $h$ is clearly coercive in $z$, there exists at least $z^\star\in \mathcal D$ maximising $h$. After tedious but easy computations, we prove that the uniqueness of the maximizer in $ \mathcal D$ of $h$ holds. Following the proof of Theorem \ref{thm:optimal:ex}, the corresponding HJB equation is 
\begin{equation}
  \left\{
      \begin{aligned}\label{HJB:exemple:variance}
       &-\partial_t v(t,\nu)-\sup_{z\in \tilde {\mathcal Z}_p}\left(\int_\mathbb R \mathbb L_t^z v(t,\nu)(x) \nu(dx) \right)=0\\
       &v(T,\nu)= \int_{\mathbb R^2} (x-y) \nu(dx,dy)+2\lambda_{X\xi} \left(\int_{\R^2} xy \nu(dx,dy)  -\int_\R x \nu^1(dx)\int_\R y \nu^2(dy) \right)\\
       &\hspace{3.8em}-(\lambda_{X}+\lambda_{X\xi}) \left(\int_\R |x|^2 \nu^1(dx) -\left(\int_\R x \nu^1(dx) \right)^2\right)\\
       &\hspace{3.8em}-(\lambda_\xi+\lambda_{X\xi}) \left(\int_\R |y|^2 \nu^2(dy) -\left(\int_\R y \nu^2(dy) \right)^2\right).
      \end{aligned}
    \right.
\end{equation}
As previously, by considering the {\it Ansatz} \ref{ansatz}, and by denoting $\underline U(t,\nu)$ the dynamic version of the problem the Principal restricted to deterministic $Z$, we can compute explicitly $\underline U(t,\nu)$ for any $(t,\nu)\in [0,T]\times \mathcal P(\mathcal C^2)$.

Similarly to the proof of Theorem \ref{thm:optimal:ex}, we show that HJB Equation \eqref{HJB:exemple:variance} is satisfied by $\underline U^P$. By adapting the second step of the proof of Theorem \ref{thm:optimal:ex}, we deduce that $\underline U^P$ is the value function of the Principal with the optimal control $z^\star\in \mathcal D$ maximising $h$.
\end{proof}

\begin{proof}[Proof of Theorem \ref{thm:NplayerMF}]
From \eqref{zstarN}, we deduce that $(i)$ holds. We now turn to the proof of $(ii)$. First recall that 
$$dX^N_t= \left(a^{N,\star}_t + B_N X^N_t\right)dt +\sigma dW_t^{N,\star},\, X_0^N=\psi^N.$$ Thus, by denoting $S_t^N:= \sum_{i=1}^N X^{N,i}_t$ for any $t\in [0,T]$ we have for any $i\in \{1,\dots,N\}$
$$dX^{N,i}_t= \left( \frac{\exp((\alpha+\beta_1) (T-t))}c + \alpha X_t^{N,i}+ \frac{\beta_1}N S_t^N\right)dt +\sigma d(W_t^{N,a^{N,\star}})^i.$$
Notice now that 
$$ d\left(\frac{S_t^N}N\right)= \left( \frac{\exp((\alpha+\beta_1) (T-t))}c +\left(\alpha+\beta_1 \right) \frac{S_t^N}N\right)dt +\sigma \frac{\mathbf 1_N dW_t^{N,a^{N,\star}}}N.$$
Then, since $\frac{S^N}N$ is an Ornstein--Uhlenbeck process, we get
\begin{align*}\frac{S^N_t}N
&=\frac1N\sum_{i=1}^N \psi^{N,i}e^{\kappa t}+\frac{1}{2\kappa c}\left(e^{\kappa(T+t)}- e^{\kappa (T-t)}\right) +\int_0^t \sigma e^{-(\alpha+\beta_1)(s-t)}\frac{\mathbf 1_N}N \cdot dW_s^{N,a^{N,\star}}.
\end{align*}
Hence, using the stochastic Fubini's theorem, we get
\begin{align*}X^{N,i}_t=
 \psi^{N,i}+\int_0^t (\alpha X_s^{N,i}+\theta^N_s)ds+\int_0^t\sigma^N_s \cdot dW_s^{N,a^{N,\star}},
\end{align*}
where for any $t\in [0,T]$
$$\theta^N_t:=  \frac{\exp(\kappa (T-t))}c + \beta_1\frac1N\sum_{i=1}^N \psi^{N,i}e^{\kappa t}+\frac{\beta_1}{2\kappa c}\left(e^{\kappa(T+t)}- e^{\kappa (T-t)}\right),\; \sigma^N_t= \frac{\sigma}\kappa(1-e^{-\kappa (s-t)}) \frac{\mathbf 1_N}N+\sigma e_i.$$
Therefore, $X^{N,i}$ is an Ornstein--Uhlenbeck process with parameters $\alpha,\theta^N_t$ and volatility $\sigma^N_t$. By setting $\theta^\star:=  \frac{\exp(\kappa (T-t))}c+ f(t)$ where $f(t)$ denotes the solution to ODE \eqref{ode:f:linear}, we deduce that the following convergences hold
$$\lambda_0\circ(\theta_t^N)^{-1}\overset{\rm weakly}{\underset{N\to+\infty}{\longrightarrow}}\lambda_0\circ(\theta^\star_t)^{-1}, $$
 and
$$\mathbb P^{a^{N,\star}}_N\circ\left(\int_0^\cdot\sigma^N_s \cdot dW_s^{N,\star}\right)^{-1}\overset{\rm weakly}{\underset{N\to+\infty}{\longrightarrow}} \mathbb P^{a^{\star}}\circ\left(\int_0^\cdot\sigma dW_s^{\star}\right)^{-1}.$$
It is then clear that the required convergence for $X^{N,i}$ holds. The fact that $\xi^{N,\star}$ is optimal for the $N$--players' model is a direct consequence of \cite{elie2016contracting,mastrolia2017moral}. Notice that for any $i\in \{1,\dots, N \}$
\begin{align*}
(\xi^{N,\star})^i
&=  R_0-\int_0^T \frac{\exp(2\kappa (T-t))}{2c}dt-\int_0^T e^{\kappa (T-t)}\left(\alpha X_t^{N,i}+\beta_1 \frac1N\sum_{i=1}^N X_t^{N,i} \right)dt+\int_0^T e^{\kappa(T-t)}dX^{N,i}_t.\end{align*}
By the above computations, it is clear that for any $t\in[0,T]$, the law of $N^{-1}S_t^N$ under $\P_N^{a^{N,\star}}$ converges weakly to a Dirac mass at $\mathbb E^\star[X^\star_t]$. Hence the desired result $(ii)$.
\end{proof}

 \bibliographystyle{plain}
 \small
\bibliography{bibliographyDylan}

 \end{document}